\documentclass[reqno,12pt]{amsart}
\usepackage{amssymb}
\usepackage{amsfonts}
\usepackage{amsmath}

\newtheorem{theorem}{Theorem}
\theoremstyle{plain}
\newtheorem{acknowledgement}{Acknowledgement}

\newtheorem{corollary}{Corollary}

\newtheorem{lemma}{Lemma}

\newtheorem{remark}{Remark}

\numberwithin{equation}{section}

\input{tcilatex}

\begin{document}
\author{}
\title{}
\maketitle

\begin{center}
\thispagestyle{empty} \pagestyle{myheadings} 
\markboth{\bf  Yilmaz Simsek,
}{\bf Formulas for p-adic q-integrals including falling and rising factorials }

\textbf{Formulas for }$p$\textbf{-adic }$q$\textbf{-integrals including
falling-rising factorials, Combinatorial sums and special numbers}

\bigskip

\textbf{Yilmaz Simsek}

\bigskip

\textit{Department of Mathematics, Faculty of Science University of Akdeniz
TR-07058 Antalya, Turkey, }

\textit{E-mail: ysimsek@akdeniz.edu.tr\\[0pt]
}

\bigskip

\bigskip

\textbf{{\large {Abstract}}}\medskip
\end{center}

The main purpose of this paper is to provide a novel approach to deriving
formulas for the $p$-adic $q$-integral including the Volkenborn integral and
the $p$-adic fermionic integral. By applying integral equations and these
integral formulas to the falling factorials, the rising factorials and
binomial coefficients, we derive some new and old identities and relations
related to various combinatorial sums, well-known special numbers such as
the Bernoulli and Euler numbers, the harmonic numbers, the Stirling numbers,
the Lah numbers, the Harmonic numbers, the Fubini numbers, the Daehee
numbers and the Changhee numbers. Applying these identities and formulas, we
give some new combinatorial sums. Finally, by using integral equations, we
derive generating functions for new families of special numbers and
polynomials. We also give further comments and remarks on these functions,
numbers and integral formulas.\bigskip 

\noindent \textbf{2010 Mathematics Subject Classification.} 11B68; 05A15;
05A19; 12D10; 26C05; 30C15.

\noindent \textbf{Key Words.} $p$-adic $q$-integrals, Volkenborn integral,
Dirichlet character, Bernoulli numbers and polynomials, Euler numbers and
polynomials, Stirling numbers, Lah numbers, Daehee numbers, Changhee
numbers, Combinatorial sum, Generating functions.

\bigskip

\section{Introduction}

This section deals with comprehensive study of analytic objects linked to
theory of the Volkenborn integral, the $p$-adic fermionic integral and the
generating functions for special numbers and polynomials. The $p$-adic
integral and generating functions have been used in mathematics, in
mathematical physics and in others sciences. Especially the $p$-adic
integral and $p$-adic numbers are used in the theory of ultrametric
calculus, the $p$-adic quantum mechanics and the $p$-adic mechanics.

In this paper, by using the Volkenborn integral and the fermionic integral
and their integral equations, we give generating functions for special
numbers and polynomials. By applying these integrals to the falling and
rising factorials with their identities and relations, we derive both to
standard and to new formulas, identities and relations closely related to
the Volkenborn integral, the fermionic integral, combinatorial sums and
special numbers. These formulas in particular will allow us to solve and
compute these integrals, including the falling and rising factorials,
numerically much more efficiently. We can simply give some definitions and
results for the $p$-adic integrals, which are detailed study in the
following references: \cite{Amice}, \cite{Khrennikov}, \cite{T. Kim}, \cite%
{Kim2006TMIC}, \cite{Schikof}, \cite{Vladimirov}, \cite{Volkenborn}; and the
references cited therein.

To state the $p$-adic $q$-Volkenborn integral, it is useful to introduce the
following notations. Let $\mathbb{Z}_{p}$ be a set of $p$-adic integers. Let 
$\ \mathbb{K}$ be a field with a complete valuation and $C^{1}(\mathbb{Z}%
_{p}\rightarrow \mathbb{K)}$ be a set of continuous derivative functions.
That is $C^{1}(\mathbb{Z}_{p}\rightarrow \mathbb{K)}$ is contained in $%
\left\{ f:\mathbb{X}\rightarrow \mathbb{K}:f(x)\text{ is differentiable and }%
\frac{d}{dx}f(x)\text{ is continuous}\right\} $.

Kim \cite{T. Kim} introduced and systematically studied the following family
of the $p$-adic $q$-integral which provides a unification of the Volkenborn
integral:%
\begin{equation}
I_{q}(f(x))=\int_{\mathbb{Z}_{p}}f(x)d\mu _{q}(x)=\lim_{N\rightarrow \infty }%
\frac{1}{[p^{N}]_{q}}\sum_{x=0}^{p^{N}-1}f(x)q^{x},  \label{q-BI}
\end{equation}%
where $q\in \mathbb{C}_{p}$, the completion of the algebraic closure of $%
\mathbb{Q}_{p}$, set of $p$-adic rational numbers, with $\left\vert
1-q\right\vert _{p}<1$, $f\in C^{1}(\mathbb{Z}_{p}\rightarrow \mathbb{K)}$,%
\begin{equation*}
\left[ x\right] =\left[ x:q\right] =\left\{ 
\begin{array}{c}
\frac{1-q^{x}}{1-q},q\neq 1 \\ 
x,q=1%
\end{array}%
\right.
\end{equation*}%
and $\mu _{q}(x)=\mu _{q}\left( x+p^{N}\mathbb{Z}_{p}\right) $ denotes the $%
q $-distribution on $\mathbb{Z}_{p}$, defined by%
\begin{equation*}
\mu _{q}\left( x+p^{N}\mathbb{Z}_{p}\right) =\frac{q^{x}}{\left[ p^{N}\right]
_{q}},
\end{equation*}%
(\textit{cf}. \cite{T. Kim}).

Let $\mathbb{X}$ be a compact-open subset of $\mathbb{Q}_{p}$. A $p$-adic
distribution $\mu $ on $\mathbb{X}$ is a $\mathbb{Q}_{p}$-linear vector
space homomorphism from the $\mathbb{Q}_{p}$-vector space of locally
constant functions on $\mathbb{X}$ to $\mathbb{Q}_{p}$.

\begin{remark}
The $p$-adic $q$-integral are related to the quantum groups, cohomology 
groups, $q$-deformed oscillator and $p$-adic models (cf. \cite{Khrennikov},  
\cite{Vladimirov}).
\end{remark}

\begin{remark}
If $q\rightarrow 1$, then (\ref{q-BI}) reduces to the Volkenborn integral. 
That is 
\begin{equation*}
\lim_{q\rightarrow 1}I_{q}(f(x))=I_{1}(f(x))
\end{equation*}
where 
\begin{equation}
I_{1}(f(x))=\int\limits_{\mathbb{Z}_{p}}f\left( x\right) d\mu _{1}\left(
x\right) =\underset{N\rightarrow \infty }{\lim }\frac{1}{p^{N}}
\sum_{x=0}^{p^{N}-1}f\left( x\right)  \label{M}
\end{equation}
and\ $\mu _{1}\left( x\right) $ denotes the Haar distribution. $I_{1}(f(x))$
is also so-called the bosonic integral (\textit{cf}. \cite{Amice}, \cite%
{Khrennikov}, \cite{Schikof}, \cite{Vladimirov}, \cite{Volkenborn}); see 
also the references cited in each of these earlier works). This integral has
been many applications not only in mathematics, but also in mathematical 
physics. By using this integral and its integral equations, various 
different generating functions related to the Bernoulli type numbers and 
polynomials have been constructed.
\end{remark}

\begin{remark}
If $q\rightarrow -1$, then (\ref{q-BI}) reduces to the $p$-adic fermionic 
integral which is defined by Kim \cite{Kim2006TMIC}. That is 
\begin{equation*}
\lim_{q\rightarrow -1}I_{q}(f(x))=I_{-1}(f(x))
\end{equation*}
where 
\begin{equation}
I_{-1}(f(x))=\int\limits_{\mathbb{Z}_{p}}f\left( x\right) d\mu _{-1}\left(
x\right) =\underset{N\rightarrow \infty }{\lim }\sum_{x=0}^{p^{N}-1}\left(
-1\right) ^{x}f\left( x\right)  \label{Mmm}
\end{equation}
and 
\begin{equation*}
\mu _{-1}\left( x+p^{N}\mathbb{Z}_{p}\right) =\frac{(-1)^{x}}{p^{N}}
\end{equation*}
(\textit{cf}. \cite{Kim2006TMIC}). By using the $p$-adic fermionic integral 
and its integral equations, various different generating functions related 
to the Euler and Genocchi type numbers and polynomials have been constructed.
\end{remark}

In order to give our paper results, we need some properties, identities and
definitions for generating functions of the special numbers and polynomials.
These functions have many valuable applications in almost all areas of
mathematics, in mathematical physics and in computer and in engineering
problems and in other areas of science.

It is well-known that the $\lambda $-Bernoulli numbers and polynomials and
the $\lambda $-Euler numbers and polynomials have been studied in different
sets. Thus, we give the following observations for the $\lambda $ parameter.
When we study generating functions on the set of complex numbers, we assume
that $\lambda \in \mathbb{C}$. When we study generating functions by the $p$%
-adic integral, we assume that $\lambda \in \mathbb{Z}_{p}$. We start with
definition of generating function for the Apostol-Bernoulli polynomials $%
\mathcal{B}_{n}(x;\lambda )$ as follows:%
\begin{equation}
F_{A}(t,x;\lambda )=\frac{t}{\lambda e^{t}-1}e^{tx}=\sum_{n=0}^{\infty }%
\mathcal{B}_{n}(x;\lambda )\frac{t^{n}}{n!}.  \label{Ap.B}
\end{equation}%
Substituting $x=0$ into (\ref{Ap.B}), we have%
\begin{equation*}
\mathcal{B}_{n}(\lambda )=\mathcal{B}_{n}(0;\lambda )
\end{equation*}%
denotes the Apostol-Bernoulli numbers (\textit{cf}. \cite{Kim2006TMIC}, \cite%
{KIMjang}, \cite{Luo}, \cite{Srivastava2011}, \cite{srivas18}; see also the
references cited in each of these earlier works), and also%
\begin{equation*}
B_{n}=\mathcal{B}_{n}(0;1)
\end{equation*}%
denotes the Bernoulli numbers of the first kind\textit{\ }(\textit{cf}. \cite%
{Bayad}-\cite{Wang}; see also the references cited in each of these earlier
works).

Kim \textit{et al}. \cite{TkimJKMS} defined the $\lambda $-Bernoulli
polynomials (Apostol-type Bernoulli polynomials) $\mathfrak{B}_{n}(x;\lambda
)$ by means of the following generating function:%
\begin{equation}
F_{B}(t,x;\lambda )=\frac{\log \lambda +t}{\lambda e^{t}-1}%
e^{tx}=\sum_{n=0}^{\infty }\mathfrak{B}_{n}(x;\lambda )\frac{t^{n}}{n!}
\label{laBN}
\end{equation}%
($\left\vert t\right\vert <2\pi $ when $\lambda =1$ and $\left\vert
t\right\vert <\left\vert \log \lambda \right\vert $ when $\lambda \neq 1$)
with%
\begin{equation*}
\mathfrak{B}_{n}(\lambda )=\mathfrak{B}_{n}(0;\lambda )
\end{equation*}%
denotes the $\lambda $-Bernoulli numbers (Apostol-type Bernoulli numbers) (%
\textit{cf}. \cite{TkimJKMS}, \cite{jandY1}, \cite{srivas18}, \cite%
{simsek2017ascm}).

The Fubini numbers $w_{g}(n)$ are defined by means of the following
generating functions:%
\begin{equation}
F_{Fu}(t)=\frac{1}{2-e^{t}}=\sum_{n=0}^{\infty }w_{g}(n)\frac{t^{n}}{n!},
\label{IR-5}
\end{equation}%
(\textit{cf}. \cite{Good}).

The Frobenius-Euler numbers $H_{n}(u)$ are defined by means of the following
generating function:

Let $u$ be a complex numbers with $u\neq 1$.%
\begin{equation*}
F_{f}(t,u)=\frac{1-u}{e^{t}-u}=\sum_{n=0}^{\infty }H_{n}(u)\frac{t^{n}}{n!}
\end{equation*}%
(\textit{cf}. \cite{DSKIMfrob}, \cite[Theorem 1, p. 439]{TkimJKMS}, \cite%
{jnt2015}, \cite{srivas18}; see also the references cited in each of these
earlier works).

The Apostol-Euler polynomials of first kind $\mathcal{E}_{n}(x,\lambda )$
are defined by means of the following generating function:%
\begin{equation}
F_{P1}(t,x;k,\lambda )=\frac{2}{\lambda e^{t}+1}e^{tx}=\sum_{n=0}^{\infty }%
\mathcal{E}_{n}(x,\lambda )\frac{t^{n}}{n!},  \label{Cad3}
\end{equation}%
($\left\vert t\right\vert <\pi $ when $\lambda =1$ and $\left\vert
t\right\vert <\left\vert \ln \left( -\lambda \right) \right\vert $ when $%
\lambda \neq 1$), $\lambda \in \mathbb{C}$. Substituting $x=0$ into (\ref%
{Cad3}), we have the first kind Apostol-Euler numbers%
\begin{equation*}
\mathcal{E}_{n}(\lambda )=\mathcal{E}_{n}(0,\lambda )
\end{equation*}

Setting $\lambda =1$ into (\ref{Cad3}), one has the first kind Euler numbers%
\begin{equation*}
E_{n}=\mathcal{E}_{n}^{(1)}(1)
\end{equation*}%
(\textit{cf}. \cite{Bayad}-\cite{SrivastavaLiu}; see also the references
cited in each of these earlier works).

The Euler numbers of the second kind $E_{n}^{\ast }$ are given by%
\begin{equation*}
E_{n}^{\ast }=2^{n}E_{n}\left( \frac{1}{2}\right)
\end{equation*}%
(\textit{cf}. \cite{Roman}, \cite{simsekP2}, \cite{SrivatavaChoi}; see also
the references cited in each of these earlier works).

The Stirling numbers of the first kind $S_{1}(n,k)$ the number of
permutations of $n$ letters which consist of \ $k$ disjoint cycles, are
defined by means of the following generating function:%
\begin{equation}
F_{S1}(t,k)=\frac{\left( \log (1+t)\right) ^{k}}{k!}=\sum_{n=0}^{\infty
}S_{1}(n,k)\frac{t^{n}}{n!}.  \label{S1}
\end{equation}%
These numbers have the following properties:

$S_{1}(0,0)=1$. $S_{1}(0,k)=0$ if $k>0$. $S_{1}(n,0)=0$ if $n>0$. $%
S_{1}(n,k)=0$ if $k>n$ and also%
\begin{equation}
S_{1}(n+1,k)=-nS_{1}(n,k)+S_{1}(n,k-1)  \label{S2-1c}
\end{equation}%
(\textit{cf}. \cite{Charamb}, \cite{Bayad}, \cite{Chan}, \cite{Roman}, \cite%
{SimsekFPTA}, \cite{AM2014}; and see also the references cited in each of
these earlier works).

The rising factorial is defined by%
\begin{eqnarray*}
x^{(n)} &=&x(x+1)(x+2)\ldots (x+n-1) \\
x^{(0)} &=&1,
\end{eqnarray*}%
and the falling factorial is defined by%
\begin{eqnarray*}
x_{(n)} &=&x(x-1)(x-2)\ldots (x-n+1) \\
x_{(0)} &=&1,
\end{eqnarray*}%
(\textit{cf}. \cite{Boyadzhiev}, \cite{Charamb}, \cite{Cigler}, \cite{Comtet}%
, \cite{SrivatavaChoi}).

A relation between he falling factorial and Stirling numbers of the first
kind is given by%
\begin{equation}
x_{(n)}=\sum_{k=0}^{n}S_{1}(n,k)x^{k}  \label{S1a}
\end{equation}%
(\textit{cf}. \cite{Charamb}, \cite{Cigler}, \cite{Comtet}, \cite%
{SrivatavaChoi}).

The Bernoulli polynomials of the second kind $b_{n}(x)$ are defined by means
of the following generating function:%
\begin{equation}
F_{b2}(t,x)=\frac{t}{\log (1+t)}(1+t)^{x}=\sum_{n=0}^{\infty }b_{n}(x)\frac{
t^{n}}{n!}  \label{Br2}
\end{equation}%
(\textit{cf}. \cite[pp. 113-117]{Roman}; see also the references cited in
each of these earlier works).

The Bernoulli numbers of the second kind $b_{n}(0)$ are defined by means of
the following generating function: 
\begin{equation}
F_{b2}(t)=\frac{t}{\log (1+t)}=\sum_{n=0}^{\infty }b_{n}(0)\frac{t^{n}}{n!}.
\label{Be-1t}
\end{equation}%
The Bernoulli polynomials of the second kind are defined by%
\begin{equation*}
b_{n}(x)=\int_{x}^{x+1}(u)_{n}du.
\end{equation*}%
Substituting $x=0$ into the above equation, one has%
\begin{equation}
b_{n}(0)=\int_{0}^{1}(u)_{n}du.  \label{LamdaFun-1p}
\end{equation}%
The numbers $b_{n}(0)$ are also so-called the \textit{Cauchy numbers} (%
\textit{cf}. \cite[p. 116]{Roman}, \cite{TKimTAKAO}, \cite{Merlini}, \cite%
{Qi}, \cite{ysimsek Ascm}; see also the references cited in each of these
earlier works).

In \cite{SimsekFPTA}, Simsek defined the $\lambda $-array polynomials $%
S_{v}^{n}(x;\lambda )$ by the following generating function:%
\begin{equation}
F_{A}(t,x,v;\lambda )=\frac{\left( \lambda e^{t}-1\right) ^{v}}{v!}%
e^{tx}=\sum_{n=0}^{\infty }S_{v}^{n}(x;\lambda )\frac{t^{n}}{n!},
\label{ARY-1}
\end{equation}%
where $v\in \mathbb{N}_{0}$ and $\lambda \in \mathbb{C}$ (\textit{cf}. \cite%
{Bayad}, \cite{Chan}, \cite{SimsekFPTA}, \cite{AM2014}; and see also the
references cited in each of these earlier works).

The second kind Stirling numbers $S_{2}(n,k;\lambda )$ are defined by means
of the following generating function:%
\begin{equation}
F_{S}(t,k;\lambda )=\frac{\left( \lambda e^{t}-1\right) ^{k}}{k!}%
=\sum_{n=0}^{\infty }S_{2}(n,k;\lambda )\frac{t^{n}}{n!},  \label{SN-1}
\end{equation}%
where $k\in \mathbb{N}_{0}$ and $\lambda \in \mathbb{C}$ (\textit{cf}. \cite%
{Luo}, \cite{SimsekFPTA}, \cite{Srivastava2011}; see also the references
cited in each of these earlier works).

Substituting $\lambda =1$ into (\ref{SN-1}), then we get the Stirling
numbers of the second kind, the number of partitions of a set of $n$
elements into $k$ nonempty subsets,%
\begin{equation*}
S_{2}(n,v)=S_{2}(n,v;1).
\end{equation*}%
The Stirling numbers of the second kind are also given by the following
generating function including the falling factorial:%
\begin{equation}
x^{n}=\sum\limits_{k=0}^{n}S_{2}(n,k)x_{(k)},  \label{S2-1a}
\end{equation}%
(\textit{cf}. \cite{Aigner}-\cite{SrivastavaLiu}; see also the references
cited in each of these earlier works).

The Schlomilch formula associated with a the Stirling numbers of the first
and the second kind numbers is given by%
\begin{equation*}
S_{1}(n,k)=\sum_{j=0}^{n-k}(-1)^{j}\left( 
\begin{array}{c}
n+j-1 \\ 
k-1%
\end{array}%
\right) \left( 
\begin{array}{c}
2n-k \\ 
n-k-j%
\end{array}%
\right) S_{2}(n-k+j,j)
\end{equation*}%
(\textit{cf}. \cite[p. 115]{C. A. CharalambidesDISCRETE}, \cite[p. 290,
Eq-(8.21)]{Charamb}).

The associated Stirling numbers of the second kind are defined by means of
the following generating function:%
\begin{equation*}
F_{S2}(t,v;\lambda )=\frac{\left( e^{t}-1-u\right) ^{k}}{k!}%
=\sum_{n=0}^{\infty }S_{22}(n,k)\frac{t^{n}}{n!},
\end{equation*}%
where $k\in \mathbb{N}_{0}$,%
\begin{equation*}
S_{2}(n,k)=\sum_{j=0}^{k}\left( 
\begin{array}{c}
k \\ 
j%
\end{array}%
\right) S_{22}(n-j,k-j)
\end{equation*}%
and $S_{22}(n,k)=0$, $k>\frac{n}{2}$\ (\textit{cf}. \cite[pp. 123-127]{C. A.
CharalambidesDISCRETE}). From the above generating function, we give the
following functional equation:%
\begin{equation*}
F_{S2}(t,v;\lambda )=\frac{1}{k!}\sum_{j=0}^{k}\left( -1\right) ^{k-j}\left( 
\begin{array}{c}
k \\ 
j%
\end{array}%
\right) j!F_{S}(t,j;1)u^{k-j}.
\end{equation*}%
From this equation, we have%
\begin{equation*}
S_{22}(n,k)=\frac{1}{k!}\sum_{j=0}^{k}\left( -1\right) ^{k-j}\left( 
\begin{array}{c}
k \\ 
j%
\end{array}%
\right) j!S_{2}(n,j)u^{k-j}.
\end{equation*}

The associated Stirling numbers of the first kind are defined by means of
the following generating function:%
\begin{equation*}
F_{S12}(t,v;\lambda )=\frac{\left( \log \left( 1+u\right) -u\right) ^{k}}{k!}%
=\sum_{n=0}^{\infty }S_{12}(n,k)\frac{t^{n}}{n!},
\end{equation*}%
where $k\in \mathbb{N}_{0}$,%
\begin{equation*}
s_{1}(n,k)=\sum_{j=0}^{k}\left( 
\begin{array}{c}
n \\ 
j%
\end{array}%
\right) S_{12}(n-j,k-j)
\end{equation*}%
and $S_{12}(n,k)=0$, $k>\frac{n}{2}$\ (\textit{cf}. \cite[pp. 123-127]{C. A.
CharalambidesDISCRETE}).

The Lah numbers was discovered by Ivo Lah in 1955. These numbers are
coefficients expressing rising factorials in terms of falling factorials (%
\textit{cf}. \cite{C. A. CharalambidesDISCRETE}, \cite{Charamb}, \cite%
{Comtet}, \cite{QiLAH}, \cite{riARDON}, \cite{WikipeLAH}). The unsigned Lah
numbers have an interesting meaning especially in combinatorics. These
numbers count the number of ways a set of $n$ elements can be partitioned
into $k$ nonempty linearly ordered subsets. These numbers are related to the
some special numbers such as the Stirling numbers the first and the second
kind. The Laguerre polynomials and the others.

The Lah numbers are defined by means of the following generating
function:The Lah numbers are defined by means of the following generating
function:%
\begin{equation}
F_{L}(t,k)=\frac{1}{k!}\left( \frac{t}{1-t}\right) ^{k}=\sum_{n=k}^{\infty
}L(n,k)\frac{t^{n}}{n!}  \label{La}
\end{equation}%
(\textit{cf}. \cite[p. 44]{riARDON}, \cite{Belbechir}, \cite{WikipeLAH}, and
the references cited therein). By using this equation, we have%
\begin{equation}
L(n,k)=(-1)^{n}\frac{n!}{k!}\left( 
\begin{array}{c}
n-1 \\ 
k-1%
\end{array}%
\right) .  \label{LAH-1a}
\end{equation}%
The unsigned Lah numbers are defined by%
\begin{equation*}
\left\vert L(n,k)\right\vert =\frac{n!}{k!}\left( 
\begin{array}{c}
n-1 \\ 
k-1%
\end{array}%
\right)
\end{equation*}%
$1\leq k\leq n$. Two recurrence relations of these numbers are given by%
\begin{equation*}
L(n+1,k)=-(n+k)L(n,k)-L(n,k-1)
\end{equation*}%
with the initial conditions%
\begin{equation*}
L(n,0)=\delta _{n,0}
\end{equation*}%
and%
\begin{equation*}
L(0,k)=\delta _{0,k},
\end{equation*}%
for all $k,n\in \mathbb{N}$ and%
\begin{equation*}
L(n,k)=\sum_{j=0}^{n}(-1)^{j}s_{1}(n,j)S_{2}(j,k)
\end{equation*}%
(cf. \cite[p. 44]{riARDON}).

Let $\mathcal{L}(n,k)$ denote the set of all distributions of $n$ balls,
labelled $1,\ldots ,n$, among $k$ unlabeled, contents-ordered boxes, with no
box left empty. Such distributions is known Laguerre configurations. If%
\begin{equation*}
\left\vert L(n,k)\right\vert =\left\vert \mathcal{L}(n,k)\right\vert ,
\end{equation*}%
then%
\begin{equation*}
\left\vert L(n,k)\right\vert =\delta _{n,0},
\end{equation*}%
all $n\in \mathbb{N}$,%
\begin{equation*}
\left\vert L(n,k)\right\vert =0
\end{equation*}%
if $0\leq n<k$ (\textit{cf}. \cite{Garsia}).

Another definition of these numbers are related to the falling factorial and
the rising factorial polynomials. Riordan \cite[p. 43]{riARDON} gave the
following definition for Lah numbers:%
\begin{equation}
\left( -x\right) _{(k)}=\sum_{k=0}^{n}L(n,k)x_{(k)}  \label{Lah}
\end{equation}%
so that%
\begin{equation*}
x_{(k)}=\sum_{k=0}^{n}L(n,k)\left( -x\right) _{(k)}
\end{equation*}%
and%
\begin{equation}
x^{(n)}=\sum_{k=0}^{n}\left\vert L(n,k)\right\vert x_{(k)}.  \label{LahLAH}
\end{equation}%
(\textit{cf}. \cite{C. A. CharalambidesDISCRETE}, \cite{Charamb}, \cite%
{Comtet}, \cite{Garsia}, \cite{QiLAH}, \cite{riARDON}, \cite{WikipeLAH}).

The Daehee numbers of the first kind and the second kind are defined by
means of the following generating functions:%
\begin{equation*}
\frac{\log (1+t)}{t}=\sum_{n=0}^{\infty }D_{n}\frac{t^{n}}{n!}
\end{equation*}%
and%
\begin{equation*}
\frac{(1+t)\log (1+t)}{t}=\sum_{n=0}^{\infty }\widehat{D}_{n}\frac{t^{n}}{n!}
\end{equation*}%
(\textit{cf}. \cite{DSkimDaehee}).%
\begin{equation*}
D_{n}=\sum_{k=0}^{n}S_{1}(n,k)B_{k}=(-1)^{n}\frac{n!}{n+1}
\end{equation*}%
(\textit{cf}. \cite[p. 45]{riARDON}, \cite{ElDosky}, \cite{DSkimDaehee}).

The Changhee numbers of the first kind and the second kind are defined by
means of the following generating functions:%
\begin{equation*}
\frac{2}{t+1}=\sum_{n=0}^{\infty }Ch_{n}\frac{t^{n}}{n!}
\end{equation*}%
and%
\begin{equation*}
\frac{2(1+t)}{t+2}=\sum_{n=0}^{\infty }\widehat{Ch}_{n}\frac{t^{n}}{n!}
\end{equation*}%
(\textit{cf}. \cite{DSkim2}).%
\begin{equation*}
Ch_{n}=\sum_{k=0}^{n}S_{1}(n,k)E_{k}=(-1)^{n}\frac{n!}{2^{n}}
\end{equation*}%
(\textit{cf}. \cite{DSkim2}).

The Peters polynomials $s_{k}(x;\lambda ,\mu )$, which are Sheffer
polynomials, are defined by means of the following generating functions: 
\begin{equation*}
\frac{1}{\left( 1+\left( 1+t\right) ^{\lambda }\right) ^{\mu }}%
(1+t)^{x}=\sum_{n=0}^{\infty }s_{k}(x;\lambda ,\mu )\frac{t^{n}}{n!}
\end{equation*}%
(\textit{cf.} \cite{Roman}, \cite{Jordan}). If $\mu =1$, then the
polynomials $s_{k}(x;\lambda ,\mu )$ are reduced to the Boole polynomials.
If $\lambda =1$ and $\mu =1$, then these polynomials are also reduced to the
Changhee polynomials (\textit{cf.} \cite{Roman}, \cite{DSKIMBoole}).

This paper can be considered as having 10 main sections.

In Section 2, we give some properties of the $p$-adic $q$-integrals and the $%
p$-adic fermionic integral with their integral equations. Using these
equations, we construct generating functions for special numbers and
polynomials, some identities and formulas including combinatorial sums. We
also give interpolation of these numbers.

In Section 3, we give some application of the Volkenborn integral to the
falling and rising factorials. We define sequences of the Bernoulli numbers
related to these applications. We give some integral formulas including the
Bernoulli numbers and polynomials, the Euler numbers and polynomials, the
Stirling numbers, the Lah numbers and the combinatorial sums.

In Section 4, we give some formulas for the sequence $Y_{1}(n:B)$, sequence
of the Bernoulli numbers. By using the Volkenborn integral and its integral
equations, we give some formula identities of the sequence $Y_{1}(n:B)$. We
also gives integral formulas including the falling factorials.

In Section 5, we also give some formulas for the sequence $Y_{2}(n:B)$.
Using the Volkenborn integral, we derive some formula identities of this
sequence and some integral formulas related to the falling factorials.

In Section 6, we give various integral formulas for the Volkenborn integral
associated with the falling factorials, the combinatorial sums, the special
numbers including the Bernoulli numbers, the Stirling numbers and the Lah
numbers.

In Section 7, we give applications of the $p$-adic fermionic integral
associated with the falling and rising factorials. We give some integral
formulas including the Euler numbers and polynomials, the Stirling numbers,
the Lah numbers and the combinatorial sums. We define two the sequences
related to the Euler numbers and the Euler polynomials.

In Section 8, by using the fermionic integral and its integral equations, we
derive some formula for of the sequence $\left( y_{1}(n:E)\right) $ and the $%
p$-adic fermionic integral formulas related to the falling factorials.

In Section 9, using the fermionic integral, we give some interesting
formulas for the sequence $\left( y_{2}(n:E)\right) $ and the $p$-adic
fermionic integral including the raising factorials.

In Section 10, We give some novel identities for combinatorial sums
including special numbers associated with the Bernoulli numbers, the Euler
numbers, the Stirling numbers, the Eulerian numbers and the Lah numbers.

\section{Integral equations for the Volkenborn integral, the $p$-adic $q$
-integrals and the fermionic integral}

In this section, we give and survey some properties of the $p$-adic $q$%
-integrals. We also give integral equations for this integrals. By using
these equations, we derive generating functions for special numbers and
polynomials, some identities and formulas including combinatorial sums.

\subsection{Properties of the $p$-adic Volkenborn integral}

Here, we give some standard notations for the Volkenborn integral with its
integral equations. The Volkenborn integral in terms of the Mahler
coefficients is given by the following formula:

let 
\begin{equation*}
f\left( x\right) =\sum\limits_{n=0}^{\infty }a_{n}\left( 
\begin{array}{c}
x \\ 
j%
\end{array}%
\right) \in C^{1}(\mathbb{Z}_{p}\rightarrow \mathbb{K)}.
\end{equation*}%
We have%
\begin{equation*}
\int\limits_{\mathbb{Z}_{p}}f\left( x\right) d\mu _{1}\left( x\right)
=\sum\limits_{n=0}^{\infty }\frac{(-1)^{n}}{n+1}a_{n}
\end{equation*}%
(\textit{cf}. \cite[p. 168-Proposition 55.3]{Schikof}). Let $f:\mathbb{Z}%
_{p}\rightarrow \mathbb{K}$ be an analytic function and $f\left( x\right)
=\sum\limits_{n=0}^{\infty }a_{n}x^{n}$ with $x\in \mathbb{Z}_{p}$. The
Volkenborn integral of this analytic function is given by%
\begin{equation*}
\int\limits_{\mathbb{Z}_{p}}\left( \sum\limits_{n=0}^{\infty
}a_{n}x^{n}\right) d\mu _{1}\left( x\right) =\sum\limits_{n=0}^{\infty
}a_{n}\int\limits_{\mathbb{Z}_{p}}x^{n}d\mu _{1}\left( x\right)
\end{equation*}%
or%
\begin{equation*}
\int\limits_{\mathbb{Z}_{p}}\left( \sum\limits_{n=0}^{\infty
}a_{n}x^{n}\right) d\mu _{1}\left( x\right) =\sum\limits_{n=0}^{\infty
}a_{n}B_{n}
\end{equation*}%
where%
\begin{equation*}
B_{n}=\int\limits_{\mathbb{Z}_{p}}x^{n}d\mu _{1}\left( x\right) ,
\end{equation*}%
which is known as the Witt's formula for the Bernoulli numbers (\textit{cf}. 
\cite{T. Kim}, \cite{Kim2006TMIC}, \cite{Schikof}; see also the references
cited in each of these earlier works).

The following integral equation for the Volkenborn integral is given by%
\begin{equation}
\int\limits_{\mathbb{Z}_{p}}E^{m}\left[ f(x)\right] d\mu _{1}\left( x\right)
=\int\limits_{\mathbb{Z}_{p}}f(x)d\mu _{1}\left( x\right)
+\sum\limits_{x=0}^{m-1}\frac{d}{dx}\left\{ f(x)\right\}  \label{vi-1}
\end{equation}%
where%
\begin{equation*}
E^{m}\left[ f(x)\right] =f(x+m)
\end{equation*}%
(\textit{cf}. \cite{T. Kim}, \cite{Kim2006TMIC}, \cite{Schikof}, \cite%
{wikipe}; see also the references cited in each of these earlier works).

The following a novel integral equation for equation (\ref{q-BI}) was given
by Kim \cite{KIMaml2008}:%
\begin{equation}
q\int_{\mathbb{Z}_{p}}E\left[ f(x)\right] d\mu _{q}(x)=\int_{\mathbb{Z}
_{p}}f(x)d\mu _{q}(x)+\frac{q^{2}-q}{\log q}f^{^{\prime }}(0)+q(q-1)f(0),
\label{q-BI.1}
\end{equation}%
where%
\begin{equation*}
f^{^{\prime }}(0)=\frac{d}{dx}\left\{ f(x)\right\} \left\vert _{x=0}\right. .
\end{equation*}

In \cite[p. 70]{Schikof}, exponential function is defined as follows:%
\begin{equation*}
e^{t}=\sum_{n=0}^{\infty }\frac{t^{n}}{n!},
\end{equation*}%
the above series convergences in region $E$, which subset of field $\mathbb{K%
}$ with $char(\mathbb{K})=0$. Let $k$ be residue class field of $\mathbb{K}$%
. If $char(k)=p$, then%
\begin{equation*}
E=\left\{ x\in \mathbb{K}:\left\vert x\right\vert <p^{\frac{1}{1-p}}\right\}
\end{equation*}%
and if $char(k)=0$, then%
\begin{equation*}
E=\left\{ x\in \mathbb{K}:\left\vert x\right\vert <1\right\} .
\end{equation*}

Let $\lambda \in \mathbb{Z}_{p}$. We define%
\begin{equation}
f(x)=\lambda ^{x}a^{xt}.  \label{LamdaFun}
\end{equation}%
Substituting (\ref{LamdaFun}) into (\ref{q-BI.1}), we get%
\begin{equation}
\int_{\mathbb{Z}_{p}}\lambda ^{x}a^{xt}d\mu _{q}(x)=\frac{q^{2}-q}{\log q}%
\left( \frac{t\log a+\log \lambda +\log q}{\lambda qa^{t}-1}\right) ,
\label{q-BI.2}
\end{equation}%
where%
\begin{equation*}
a\in \mathbb{C}_{p}^{+}=\left\{ x\in \mathbb{C}_{p}:\left\vert
1-x\right\vert _{p}<1\right\}
\end{equation*}%
$a\neq 1$.

Substituting $\lambda =1$, we get Equation (2.5) in \cite{ODS}.

If $t=1$ and $q\rightarrow 1$ into (\ref{q-BI.2}), we have%
\begin{equation*}
\int_{\mathbb{Z}_{p}}\lambda ^{x}a^{x}d\mu _{1}(x)=\frac{\log \lambda +\log
(a)}{\lambda a-1}.
\end{equation*}%
Substituting $\lambda =1$ into the above equation, we arrive at an Exercise
55A-1 in \cite[p. 170]{Schikof} as follows:%
\begin{equation*}
\int_{\mathbb{Z}_{p}}a^{x}d\mu _{1}(x)=\frac{\log (a)}{a-1},
\end{equation*}%
where $a\in \mathbb{C}_{p}^{+}$ and $a\neq 1$.

\begin{remark}
Substituting $a=e$, $\lambda =1$ and $q\rightarrow 1$ into (\ref{q-BI.2}), 
we arrive at equation (\ref{laBN}). Substituting $a=e$ and $q\rightarrow 1$ 
into (\ref{q-BI.2}), we get Exercise 55A-2 in \cite[p. 170]{Schikof} , which
gives us generating function for the Bernoulli numbers $B_{n}$: 
\begin{equation*}
\int_{\mathbb{Z}_{p}}e^{tx}d\mu _{1}(x)=\frac{t}{e^{t}-1}
\end{equation*}
where $t\in E$ and $t\neq 0$.
\end{remark}

By using (\ref{q-BI.2}), we define a new family of numbers, $\mathfrak{S}%
_{n}(a;\lambda ,q)$ by means of the following generating function:%
\begin{equation}
H(t;\lambda ;a,q)=\frac{q^{2}-q}{\log q}\left( \frac{t\log a+\log \lambda
+\log q}{\lambda qa^{t}-1}\right) =\sum_{n=0}^{\infty }\mathfrak{S}%
_{n}(a;\lambda ,q)\frac{t^{n}}{n!}.  \label{q-BI.3}
\end{equation}%
The numbers $\mathfrak{S}_{n}(a;\lambda ,q)$ are related to Apostol-type
Bernoulli numbers with special values of parameters $a,\lambda $ and $q$.

We define a new family of polynomials, $\mathfrak{C}_{n}(x;a,b;\lambda ,q)$
as follows:%
\begin{equation*}
G(t,x;\lambda ;a,b,q)=b^{tx}H(t;\lambda ;a,q)=\sum_{n=0}^{\infty }\mathfrak{C%
}_{n}(x;a,b;\lambda ,q)\frac{t^{n}}{n!}.
\end{equation*}%
A Witt's type formula for the numbers $\mathfrak{S}_{n}(a;\lambda ,q)$ as
follows:

\begin{theorem}
\begin{equation}
\mathfrak{S}_{n}(a;\lambda ,q)=\int_{\mathbb{Z}_{p}}\lambda ^{x}\left( x\log
a\right) ^{n}d\mu _{q}(x).  \label{q-BI.4}
\end{equation}
\end{theorem}

\begin{remark}
Setting $q\rightarrow 1$ and $a=e$ in (\ref{q-BI.4}), we arrive at the 
Witt's formula for the Bernoulli numbers: 
\begin{equation}
\int\limits_{\mathbb{Z}_{p}}x^{n}d\mu _{1}\left( x\right) =B_{n}.  \label{M1}
\end{equation}
We also easily see that 
\begin{equation}
\int\limits_{\mathbb{Z}_{p}}\left( z+x\right) ^{n}d\mu _{1}\left( x\right)
=B_{n}(z)  \label{wb}
\end{equation}
(\textit{cf}. \cite{T. Kim}, \cite{Kim2006TMIC}, \cite{ODS}, \cite{Schikof};
see also the references cited in each of these earlier works).
\end{remark}

Let $f\in C^{1}(\mathbb{Z}_{p}\rightarrow \mathbb{K)}$. Kim \cite[Theorem 1]%
{KIMaml2008} gave the following integral equation:%
\begin{equation}
q^{n}\int\limits_{\mathbb{Z}_{p}}E^{n}\left[ f\left( x\right) \right] d\mu
_{q}\left( x\right) -\int\limits_{\mathbb{Z}_{p}}f\left( x\right) d\mu
_{q}\left( x\right) =\frac{q-1}{\log q}\left(
\sum_{j=0}^{n-1}q^{j}f^{^{\prime }}(j)+\log
q\sum_{j=0}^{n-1}q^{j}f(j)\right) ,  \label{TKint.q}
\end{equation}%
where $n$ is a positive integer and%
\begin{equation*}
f^{^{\prime }}(j)=\frac{d}{dx}\left\{ f(x)\right\} \left\vert _{x=j}\right. .
\end{equation*}

\begin{theorem}
\label{TheoremShcrikof} Let $n$ be a positive integer. 
\begin{equation}
\int\limits_{\mathbb{Z}_{p}}\left( 
\begin{array}{c}
x \\ 
n%
\end{array}
\right) d\mu _{1}\left( x\right) =\frac{(-1)^{n}}{n+1}.  \label{C7}
\end{equation}
\end{theorem}

We modify (\ref{C7}) as follows:%
\begin{equation*}
\int\limits_{\mathbb{Z}_{p}}\left( x\right) _{n}d\mu _{1}\left( x\right) =%
\frac{(-1)^{n}n!}{n+1},
\end{equation*}%
which is known as the Daehee numbers \cite{DSkimDaehee}. Theorem \ref%
{TheoremShcrikof} was proved by Schikhof \cite{Schikof}.

By using (\ref{C7}), we have%
\begin{equation}
\int\limits_{\mathbb{Z}_{p}}\left( 
\begin{array}{c}
x+n-1 \\ 
n%
\end{array}%
\right) d\mu _{1}\left( x\right) =\sum_{m=0}^{n}(-1)^{m}\left( 
\begin{array}{c}
n-1 \\ 
n-m%
\end{array}%
\right) \frac{1}{m+1}  \label{C0}
\end{equation}%
(\textit{cf}. \cite{DSkim2}, \cite{AM2014}, \cite{DSkimDaehee}). By using (%
\ref{C0}), we obtain%
\begin{equation}
\int\limits_{\mathbb{Z}_{p}}\left( x+n-1\right) _{(n)}d\mu _{1}\left(
x\right) =n!\sum_{m=0}^{n}(-1)^{m}\left( 
\begin{array}{c}
n-1 \\ 
n-m%
\end{array}%
\right) \frac{1}{m+1}.  \label{1BI}
\end{equation}

\subsection{Properties of the $p$-adic fermionic integral}

Here, we introduce some properties the fermionic $p$-adic integral with its
integral equations. By using these equations, we give generating function
for special numbers and polynomials. We also derive interpolation function
for these numbers.

Let $f\in C^{1}(\mathbb{Z}_{p}\rightarrow \mathbb{K)}$. Kim \cite%
{KIMjmaa2017} gave the following integral equation for the $p$-adic
fermionic integral on $\mathbb{Z}_{p}$ as follows:%
\begin{equation}
\int\limits_{\mathbb{Z}_{p}}E^{n}\left[ f\left( x\right) \right] d\mu
_{-1}\left( x\right) +(-1)^{n+1}\int\limits_{\mathbb{Z}_{p}}f\left( x\right)
d\mu _{-1}\left( x\right) =2\sum_{j=0}^{n-1}(-1)^{n-1-j}f(j),  \label{TKint}
\end{equation}%
where $n$ is a positive integer. Substituting $n=1$ into (\ref{TKint}), we
have the following very useful integral equation, which is used to construct
generating functions associated with Euler type numbers and polynomials:%
\begin{equation}
\int\limits_{\mathbb{Z}_{p}}f\left( x+1\right) d\mu _{-1}\left( x\right)
+\int\limits_{\mathbb{Z}_{p}}f\left( x\right) d\mu _{-1}\left( x\right)
=2f(0)  \label{MmmA}
\end{equation}%
(\textit{cf}. \cite{KIMjmaa2017}). By using (\ref{Mmm}) and (\ref{MmmA}),
the Witt's formula for the Euler numbers and polynomials are given as
follows, respectively%
\begin{equation}
\int\limits_{\mathbb{Z}_{p}}x^{n}d\mu _{-1}\left( x\right) =E_{n}
\label{Mm1}
\end{equation}%
and%
\begin{equation}
\int\limits_{\mathbb{Z}_{p}}\left( z+x\right) ^{n}d\mu _{-1}\left( x\right)
=E_{n}(z)  \label{we}
\end{equation}%
(\textit{cf}. \cite{Kim2006TMIC}, \cite{KIMjang}; see also the references
cited in each of these earlier works).

\begin{theorem}
\label{ThoremKIM} Let $n$ be a positive integer. 
\begin{equation}
\int\limits_{\mathbb{Z}_{p}}\left( 
\begin{array}{c}
x \\ 
n%
\end{array}
\right) d\mu _{-1}\left( x\right) =(-1)^{n}2^{-n}.  \label{est-3}
\end{equation}
\end{theorem}

Observe that equation (\ref{est-3}) modify as follows:%
\begin{equation*}
\int\limits_{\mathbb{Z}_{p}}\left( x\right) _{n}d\mu _{-1}\left( x\right)
=(-1)^{n}2^{-n}n!,
\end{equation*}%
which is known as the Changhee numbers \cite{DSkim2}. Theorem \ref{ThoremKIM}
was proved by Kim \textit{et al.} \cite[Theorem 2.3]{DSkim2}.

\begin{equation}
\int\limits_{\mathbb{Z}_{p}}\left( 
\begin{array}{c}
x+n-1 \\ 
n%
\end{array}%
\right) d\mu _{-1}\left( x\right) =\sum_{m=0}^{n}(-1)^{m}\left( 
\begin{array}{c}
n-1 \\ 
n-m%
\end{array}%
\right) 2^{-m}  \label{Ca-1}
\end{equation}%
(\textit{cf}. \cite{DSkim2}, \cite{AM2014}, \cite{DSkimDaehee}). By using (%
\ref{Ca-1}), we get%
\begin{equation}
\int\limits_{\mathbb{Z}_{p}}\left( x+n-1\right) _{(n)}d\mu _{-1}\left(
x\right) =n!\sum_{m=0}^{n}(-1)^{m}\left( 
\begin{array}{c}
n-1 \\ 
n-m%
\end{array}%
\right) 2^{-m}.  \label{1FI}
\end{equation}

In \cite{KIMaml2008}, Kim also gave $q$-analogies of integral equation in (%
\ref{TKint}) as follows:%
\begin{equation}
q^{d}\int\limits_{\mathbb{Z}_{p}}E^{d}f\left( x\right) d\mu _{-q}\left(
x\right) +\int\limits_{\mathbb{Z}_{p}}f\left( x\right) d\mu _{-q}\left(
x\right) =\left[ 2\right] \sum_{j=0}^{d-1}(-1)^{j}q^{j}f(j),
\label{TkimFint}
\end{equation}%
where $d$ is an positive odd integer.

Substituting (\ref{LamdaFun}) into (\ref{TkimFint}), we get%
\begin{equation}
\int\limits_{\mathbb{Z}_{p}}\left( \lambda a^{t}\right) ^{x}d\mu _{-q}\left(
x\right) =\frac{\left[ 2\right] }{\left( \lambda qa^{t}\right) ^{d}+1}%
\sum_{j=0}^{d-1}(-1)^{j}\left( \lambda qa^{t}\right) ^{j}.
\label{LamdaFun-1}
\end{equation}%
Setting $d=1$ into the above equation, we have%
\begin{equation*}
\int\limits_{\mathbb{Z}_{p}}\left( \lambda a^{t}\right) ^{x}d\mu _{-q}\left(
x\right) =\frac{\left[ 2\right] }{\lambda qa^{t}+1}.
\end{equation*}%
By using (\ref{LamdaFun-1}, we define a new family of numbers, $%
k_{n}(a;\lambda ,q)$\ by the following generating function:%
\begin{equation*}
H(a;\lambda ,q)=\frac{\left[ 2\right] }{\left( \lambda qa^{t}\right) ^{d}+1}%
\sum_{j=0}^{d-1}(-1)^{j}\left( \lambda qa^{t}\right) ^{j}=\sum_{n=0}^{\infty
}k_{n}(a;\lambda ,q)\frac{t^{n}}{n!}.
\end{equation*}%
If $d$ as odd integer, we the above generating function reduces to%
\begin{equation*}
H(a;\lambda ,q)=\left[ 2\right] \sum_{m=0}^{\infty
}\sum_{j=0}^{d-1}(-1)^{dm+j}\left( \lambda qa^{t}\right)
^{dm+j}=\sum_{v=0}^{\infty }(-1)^{v}\left( \lambda qa^{t}\right) ^{v}
\end{equation*}%
Here we study on set of complex numbers. From the above equation, we get the
following theorem:

\begin{theorem}
We assume that $\lambda $ and $q$ are complex numbers with $\left\vert 
\lambda q\right\vert <1$. Then we have 
\begin{equation}
k_{n}(a;\lambda ,q)=\sum_{v=0}^{\infty }(-1)^{v}\left( \lambda q\right)
^{v}\left( v\ln a\right) ^{n}.  \label{1BIA}
\end{equation}
\end{theorem}

Now, by using (\ref{1BIA}), we also assume that $s$ is a complex numbers
with positive real part. We define an interpolation function for the numbers 
$k_{n}(a;\lambda ,q)$ as follows:%
\begin{equation*}
\zeta (s;a;\lambda ,q)=\sum_{v=1}^{\infty }(-1)^{v}\frac{\left( \lambda
q\right) ^{v}}{\left( v\ln a\right) ^{s}}.
\end{equation*}%
Substituting $a=e$ and $\lambda =1$ into the above equation, we have%
\begin{equation*}
\zeta (s;e;1,q)=\sum_{v=1}^{\infty }(-1)^{v}\frac{q^{v}}{v^{s}}.
\end{equation*}%
If $q$ goes to $1$, the function $\zeta (s;e;1,q)$ reduces to the following
interpolation function for the Euler numbers:%
\begin{equation*}
\zeta (s;e;1,q)=\sum_{v=1}^{\infty }\frac{(-1)^{v}}{v^{s}}
\end{equation*}%
(\textit{cf}. \cite{Kim2006TMIC}, \cite{TkimJKMS}, \cite{MSKIM}, \cite%
{SrivatavaChoi}).

If $q$ goes to $-1$, the function $\zeta (s;e;1,q)$ reduces to the following
interpolation function for the Bernoulli numbers:%
\begin{equation*}
\zeta (s)=\zeta (s;e;1,q)=\sum_{v=1}^{\infty }\frac{1}{v^{s}},
\end{equation*}%
which is the Riemann zeta function (\textit{cf}. \cite{Kim2006TMIC}, \cite%
{TkimJKMS}, \cite{MSKIM}, \cite{SrivatavaChoi}).

\section{Application of the Volkenborn integral to the falling and rising
factorials}

In this section, we give applications of the Volkenborn integral on $\mathbb{%
Z}_{p}$ to the falling and rising factorials, we derive some integral
formulas including the Bernoulli numbers and polynomials, the Euler numbers
and polynomials, the Stirling numbers, the Lah numbers and the combinatorial
sums.

By using the same spirit of the Bernoulli polynomial of the second kind
which are also called Cauchy numbers of the first kind, by applying the
Volkenborn integral to the rising factorial and the falling factorial,
respectively, we derive various formulas, identities, relations and
combinatorial sums including the Bernoulli numbers, the Stirling numbers,
the Lah numbers, the Daehee numbers and the Changhee numbers.

In \cite{simsek2017ascm}, similar to the Cauchy numbers defined aid of the
Riemann integral, we studied the Bernoulli numbers sequences by using $p$%
-adic integral. Let $x_{j}\in \mathbb{Z}$ and $j\in \left\{ 1,2,\ldots
,n-1\right\} $ with $n>1$.

We define a sequence $(Y_{1}(n:B))$, including the Bernoulli numbers as
follows:%
\begin{equation}
Y_{1}(n:B)=B_{n}+\sum_{j=1}^{n-1}(-1)^{j}x_{j}B_{n-j}.  \label{Y11-a}
\end{equation}%
Kim \textit{et al}. \cite{DSkimDaehee} defined the Daehee numbers (of the
first kind) by the following integral representation as follows:%
\begin{equation}
D_{n}=\int\limits_{\mathbb{Z}_{p}}t_{(n)}d\mu _{1}\left( t\right) .
\label{Y1}
\end{equation}

Combining (\ref{Y11-a}) with (\ref{Y1}), a explicit formula for the sequence
of $(Y_{1}(n:B))$ is given by%
\begin{equation}
Y_{1}(n:B)=D_{n}  \label{IR}
\end{equation}%
Few values of (\ref{Y11-a}) is computed by (\ref{Y1}) as follows:%
\begin{eqnarray*}
Y_{1}(0 &:&B)=B_{0} \\
Y_{1}(1 &:&B)=B_{1} \\
Y_{1}(2 &:&B)=B_{2}-B_{1} \\
Y_{1}(3 &:&B)=B_{3}-3B_{2}+2B_{1} \\
Y_{1}(4 &:&B)=B_{4}-6B_{3}+11B_{2}-6B_{1},\ldots .
\end{eqnarray*}%
By using the Bernoulli numbers of the first kind, the above table reduces to
the following numerical values:%
\begin{equation*}
Y_{1}(0:B)=1,Y_{1}(1:B)=-\frac{1}{2},Y_{1}(2:B)=\frac{2}{3},Y_{1}(3:B)=-%
\frac{3}{2},Y_{1}(4:B)=\frac{24}{5},\cdots
\end{equation*}%
Due to work of Kim \textit{et al}. \cite{DSkimDaehee}, we see that $x_{j}$
coefficients computed by the Stirling numbers of the first kind.

We define a sequence $(Y_{2}(n:B))$, including the Bernoulli numbers as
follows:%
\begin{equation}
Y_{2}(n:B)=B_{n}+\sum_{j=1}^{n-1}x_{j}B_{n-j},  \label{Y22-a}
\end{equation}%
Kim \textit{et al}. \cite{DSkimDaehee} defined the Daehee numbers (of the
second kind) by the following integral representation as follows:%
\begin{equation}
\widehat{D_{n}}=\int\limits_{\mathbb{Z}_{p}}t^{(n)}d\mu _{1}\left( t\right) .
\label{Y2}
\end{equation}%
Combining (\ref{Y22-a}) with (\ref{Y2}), a explicit formula for the sequence
of $(Y_{1}(n:B))$ is given by%
\begin{equation*}
Y_{2}(n:B)=\widehat{D_{n}}
\end{equation*}%
Few values of (\ref{Y22-a}) is computed by (\ref{Y2}) as follows:%
\begin{eqnarray*}
Y_{2}(0 &:&B)=B_{0}, \\
Y_{2}(1 &:&B)=B_{1} \\
Y_{2}(2 &:&B)=B_{2}+B_{1} \\
Y_{2}(3 &:&B)=B_{3}+3B_{2}+2B_{1} \\
Y_{2}(4 &:&B)=B_{4}+6B_{3}+11B_{2}+6B_{1},\ldots .
\end{eqnarray*}%
By using the Bernoulli numbers of the first kind, the above table reduces to
the following numerical values:%
\begin{equation*}
Y_{2}(0:B)=1,Y_{2}(1:B)=-\frac{1}{2},Y_{2}(2:B)=-\frac{1}{3},Y_{2}(3:B)=-%
\frac{1}{2},Y_{2}(4:B)=-\frac{6}{5},\ldots .
\end{equation*}

Kim \textit{et al}. \cite{DSkimDaehee} defined the first and second kind
Daehee polynomials, respectively, as follows:%
\begin{equation}
D_{n}(x)=\int\limits_{\mathbb{Z}_{p}}\left( x+t\right) _{(n)}d\mu _{1}\left(
t\right)  \label{Y1a}
\end{equation}%
and%
\begin{equation}
\widehat{D_{n}}(x)=\int\limits_{\mathbb{Z}_{p}}\left( x+t\right) ^{(n)}d\mu
_{1}\left( t\right) .  \label{Y2a}
\end{equation}

We also define the following sequences for the Bernoulli polynomials, $%
(Y_{1}(x,n:B(x)))$ and $(Y_{2}(x,n:B(x)))$\ related to the polynomials $%
D_{n}(x)$\textit{\ }and $\widehat{D_{n}}(x)$ as follows:%
\begin{equation*}
Y_{1}(x,n:B(x))=B_{n}(x)+\sum_{j=1}^{n-1}(-1)^{j}x_{j}B_{n-j}(x)
\end{equation*}%
and%
\begin{equation*}
Y_{2}(x,n:B(x))=B_{n}(x)+\sum_{j=1}^{n-1}x_{j}B_{n-j}(x).
\end{equation*}

Observe that%
\begin{equation*}
Y_{1}(n:B)=Y_{1}(0,n:B(0))
\end{equation*}

and%
\begin{equation*}
Y_{2}(n:B)=Y_{2}(0,n:B(0)).
\end{equation*}

\section{Formulas for the sequence $Y_{1}(n:B)$}

To use the Volkenborn integral and its integral equations, we give some
formula identities of the sequence $\left( Y_{1}(n:B)\right) $. We also
gives some $p$-adic integral formulas including the falling factorials.

Explicit formula for the sequence $\left( Y_{1}(n:B)\right) $ is given by
the following theorem, which was proved by different method (cf. \cite%
{DSkimDaehee}, \cite{ElDosky}, \cite{ElDosky2}, \cite[p. 117]{riARDON}, \cite%
{simsekCogent}, \cite{simsek2017ascm}).

\begin{theorem}
\begin{equation}
Y_{1}(n:B)=(-1)^{n}\frac{n!}{n+1}.  \label{YY1a}
\end{equation}
\end{theorem}

\begin{proof}
We know that numbers of the sequence $\left( Y_{1}(n:B)\right) $ are related
ot the numbers $D_{n} $. By using same computaion of the numbers $D_{n} $,
this theorem is proved. That is, we now briefly give this proof. Since  
\begin{equation*}
x_{(n)}=n!\left( 
\begin{array}{c}
x \\ 
j%
\end{array}
\right) ,
\end{equation*}
by using (\ref{C7}), we get 
\begin{eqnarray}
\int\limits_{\mathbb{Z}_{p}}x_{(n)}d\mu _{1}\left( x\right)
&=&n!\int\limits_{\mathbb{Z}_{p}}\left( 
\begin{array}{c}
x \\ 
n%
\end{array}
\right) d\mu _{1}\left( x\right)  \label{FF-1} \\
&=&\frac{(-1)^{n}}{n+1}n!.  \notag
\end{eqnarray}
Thus, we get the desired result.
\end{proof}

In \cite[Theorem 1 and Eq-(2.10)]{DSkim2}, Kim defined%
\begin{eqnarray}
\int\limits_{\mathbb{Z}_{p}}x_{(n)}d\mu _{1}\left( x\right) &=&D_{n}
\label{Da-0TK} \\
&=&\sum_{l=0}^{n}s_{1}(n,l)B_{l}.  \notag
\end{eqnarray}%
By combining (\ref{FF-1}) and (\ref{Da-0TK}), we get the following result:

\begin{corollary}
\begin{equation}
\sum_{l=0}^{n}s_{1}(n,l)B_{l}=\frac{(-1)^{n}}{n+1}n!.  \label{Da-0TKa}
\end{equation}
\end{corollary}

\begin{remark}
Proof of (\ref{Da-0TKa}) is also given by Riordan \cite[p. 117]{riARDON}. 
See also (\textit{cf}. \cite{DSkimDaehee}, \cite{ElDosky}, \cite{ElDosky2},  
\cite{simsekCogent}, \cite{simsek2017ascm}). That is, equation (\ref{IR}) 
holds true.
\end{remark}

\begin{theorem}
\begin{equation*}
\int\limits_{\mathbb{Z}_{p}}\left( x+1\right) _{(n)}d\mu _{1}\left( x\right)
=(-1)^{n+1}\frac{n!}{n^{2}+n}.
\end{equation*}
\end{theorem}

\begin{proof}
Since 
\begin{equation}
\Delta x_{(n)}=\left( x+1\right) _{(n)}-x_{(n)},  \label{FF-11a}
\end{equation}
we have 
\begin{equation*}
\left( x+1\right) _{(n)}=x_{(n)}+nx_{(n-1)}
\end{equation*}
(\textit{cf}. \cite[p. 58]{Roman}). By applying the Volkenborn integral on $ 
\mathbb{Z}_{p}$ to the above equation and combining with (\ref{FF-1}), we 
get the desired result.
\end{proof}

By applying the Volkenborn integral on $\mathbb{Z}_{p}$ to the above
equation and combining with (\ref{FF-1}), we arrive at the following result:

\begin{corollary}
\begin{equation*}
\int\limits_{\mathbb{Z}_{p}}\Delta x_{(n)}d\mu _{1}\left( x\right)
=(-1)^{n+1}\left( n-1\right) !.
\end{equation*}
\end{corollary}

By applying the Volkenborn integral on $\mathbb{Z}_{p}$ to equation (\ref%
{Lah}), we obtain%
\begin{equation*}
\int\limits_{\mathbb{Z}_{p}}\left( -x\right) _{(k)}d\mu _{1}\left( x\right)
=\sum_{k=0}^{n}L(n,k)\int\limits_{\mathbb{Z}_{p}}x_{(k)}d\mu _{1}\left(
x\right) .
\end{equation*}%
By using (\ref{FF-1}), we get%
\begin{equation*}
\int\limits_{\mathbb{Z}_{p}}\left( -x\right) _{(k)}d\mu _{1}\left( x\right)
=\sum_{k=0}^{n}(-1)^{k}\frac{k!L(n,k)}{k+1}.
\end{equation*}%
Substituting (\ref{LAH-1a}) into the above equation, we arrive at the
following theorem:

\begin{theorem}
\begin{equation*}
\int\limits_{\mathbb{Z}_{p}}\left( -x\right) _{(k)}d\mu _{1}\left( x\right)
=\sum_{k=0}^{n}(-1)^{k+n}\left( 
\begin{array}{c}
n-1 \\ 
k-1%
\end{array}
\right) \frac{n!}{k+1}.
\end{equation*}
\end{theorem}

\section{Formulas for the sequence $Y_{2}(n:B)$}

To use the Volkenborn integral and its integral equations, we also give some
formula identities of the sequence $\left( Y_{2}(n:B)\right) $. We gives
some $p$-adic integral formulas including the raising factorials.

We give another formula for the numbers $Y_{2}(n:B)$ by the following
theorem:

\begin{theorem}
\begin{equation}
Y_{2}(n:B)=\sum_{k=0}^{n}C(n,k)B_{k},  \label{Y2B}
\end{equation}
where $C(n,k)=\left\vert s_{1}(n,k)\right\vert $ and $B_{k}$ denotes the 
Bernoulli numbers.
\end{theorem}

\begin{proof}
By applying the Volkenborn integral on $\mathbb{Z}_{p}$ to  the following
equation 
\begin{equation}
x^{(n)}=\sum_{k=1}^{n}C(n,k)x^{k}  \label{AY-2}
\end{equation}
and using (\ref{M1}), we get the desired result.
\end{proof}

\begin{theorem}
\begin{equation}
Y_{2}(n:B)=\sum_{k=1}^{n}(-1)^{k}\frac{n!}{k+1}\left( 
\begin{array}{c}
n-1 \\ 
k-1%
\end{array}
\right) .  \label{Y2A}
\end{equation}
\end{theorem}

\begin{proof}
By applying the Volkenborn integral on $\mathbb{Z}_{p}$ to  the equation (%
\ref{LahLAH}), we get 
\begin{equation*}
Y_{2}(n:B)=\sum_{k=1}^{n}\left\vert L(n,k)\right\vert \int\limits_{\mathbb{Z}
_{p}}x_{(k)}d\mu _{1}\left( x\right)
\end{equation*}
By substituting (\ref{FF-1}), we get the desired result.
\end{proof}

By applying the Volkenborn integral on $\mathbb{Z}_{p}$ to (\ref{LahLAH}),
we get the following formula for the numbers $Y_{2}(n:B):$

\begin{corollary}
\begin{equation}
Y_{2}(n:B)=\sum_{k=0}^{n}(-1)^{k}\frac{\left\vert L(n,k)\right\vert k!}{k+1}.
\label{S1bb}
\end{equation}
\end{corollary}

Substituting (\ref{S1a}) into (\ref{LahLAH}), we have%
\begin{equation}
x^{(n)}=\sum_{k=0}^{n}\left\vert L(n,k)\right\vert
\sum_{j=0}^{k}S_{1}(k,j)x^{j}.  \label{IR-3}
\end{equation}%
By applying the Volkenborn integral on $\mathbb{Z}_{p}$ to the above
equation, we arrive at the following formula for the numbers $Y_{2}(n:B)$:

\begin{theorem}
\begin{equation*}
Y_{2}(n:B)=\sum_{k=0}^{n}\sum_{j=0}^{k}\left\vert L(n,k)\right\vert
S_{1}(k,j)B_{j}.
\end{equation*}
\end{theorem}

By combining (\ref{Da-0TKa}) and (\ref{S1bb}), we get the following results:

\begin{equation*}
Y_{2}(n:B)=\sum_{k=0}^{n}\left\vert L(n,k)\right\vert D_{k}
\end{equation*}

\begin{lemma}
\begin{equation}
\int\limits_{\mathbb{Z}_{p}}x.x^{_{^{(n)}}}d\mu _{1}\left( x\right)
=\sum_{k=1}^{n}(-1)^{k+1}\left( 
\begin{array}{c}
n-1 \\ 
k-1%
\end{array}
\right) \frac{n!}{k^{2}+3k+2}.  \label{LL-1a}
\end{equation}
\end{lemma}

\begin{proof}
By applying the Volkenborn integral on $\mathbb{Z}_{p}$ to  the following
equation: 
\begin{equation*}
x^{_{^{(n+1)}}}=x.x^{_{^{(n)}}}+nx^{_{^{(n)}}},
\end{equation*}
and using (\ref{Y2A}), after some elementary calculations, we get the 
desired result.
\end{proof}

\begin{lemma}
\begin{equation}
\int\limits_{\mathbb{Z}_{p}}x.x^{_{^{(n)}}}d\mu _{1}\left( x\right)
=\sum_{k=1}^{n}C(n,k)B_{k+1}.  \label{LL-1b}
\end{equation}
\end{lemma}

\begin{proof}
By applying the Volkenborn integral on $\mathbb{Z}_{p}$ to  the following
equation 
\begin{equation}
xx^{(n)}=\sum_{k=1}^{n}C(n,k)x^{k+1}  \label{LL-1c}
\end{equation}
we get 
\begin{equation*}
\int\limits_{\mathbb{Z}_{p}}x.x^{_{^{(n)}}}d\mu _{1}\left( x\right)
=\sum_{k=1}^{n}C(n,k)\int\limits_{\mathbb{Z}_{p}}x^{_{^{k+1}}}d\mu
_{1}\left( x\right) .
\end{equation*}
By combining the above equation with (\ref{M1}), we get the desired result.
\end{proof}

A recurrence relation of the numbers $Y_{2}(n:B)$ is given by the following
theorem.

\begin{theorem}
\begin{equation}
Y_{2}(n+1:B)-nY_{2}(n:B)=\sum_{k=1}^{n}(-1)^{k+1}\left\vert
L(n,k)\right\vert \frac{k!}{k^{2}+3k+2}.  \label{LL-1d}
\end{equation}
\end{theorem}

\begin{proof}
We set 
\begin{equation*}
(x+n)x^{(n)}=\sum_{k=1}^{n}\left\vert L(n,k)\right\vert
xx_{(k)}+n\sum_{k=1}^{n}\left\vert L(n,k)\right\vert x_{(k)}.
\end{equation*}
From the equation, we get 
\begin{equation*}
x^{(n+1)}=\sum_{k=1}^{n}\left\vert L(n,k)\right\vert
xx_{(k)}+n\sum_{k=1}^{n}\left\vert L(n,k)\right\vert x_{(k)}
\end{equation*}
By applying the Volkenborn integral on $\mathbb{Z}_{p}$ to  the above
equation, and using (\ref{L1}), we get 
\begin{equation*}
Y_{2}(n+1:B)=\sum_{k=1}^{n}(-1)^{k+1}\left\vert L(n,k)\right\vert \frac{k!}{
k^{2}+3k+2}+nY_{2}(n:B).
\end{equation*}
Therefore, the proof of the theorem is completed.
\end{proof}

\section{Integral formulas for the Volkenborn integral}

In this section, we give some integral formulas for the Volkenborn integral
including the falling factorials, the combinatorial sums, the special
numbers including the Bernoulli numbers, the Stirling numbers and the Lah
numbers.

\begin{lemma}
\begin{equation}
\int\limits_{\mathbb{Z}_{p}}x.x_{(n)}d\mu _{1}\left( x\right) =(-1)^{n+1} 
\frac{n!}{n^{2}+3n+2}.  \label{L1}
\end{equation}
\end{lemma}

\begin{proof}
Since 
\begin{equation}
x.x_{(n)}=x_{(n+1)}+nx_{(n)}.  \label{Ro}
\end{equation}
By applying the Volkenborn integral on $\mathbb{Z}_{p}$ to  the both sides
of the above equation, and using (\ref{C7}), we arrive at the  desired
result.
\end{proof}

By combining (\ref{Ro}) with (\ref{S1a}), we have%
\begin{equation}
x.x_{(n)}=\sum_{k=0}^{n}\left( S_{1}(n+1,k)+nS_{1}(n,k)\right) x^{k}+x^{n+1}.
\label{LamdaFun-A}
\end{equation}%
By combining the above equation with (\ref{S2-1c}), since $k<0$, $%
S_{1}(n,k)=0$, we get%
\begin{equation*}
x.x_{(n)}=\sum_{k=1}^{n}S_{1}(n,k-1)x^{k}+x^{n+1}.
\end{equation*}%
By applying the Volkenborn integral on $\mathbb{Z}_{p}$ to the both sides of
the above equation, and using (\ref{M1}), we also arrive at the following
lemma:

\begin{lemma}
\begin{equation}
\int\limits_{\mathbb{Z}_{p}}x.x_{(n)}d\mu _{1}\left( x\right)
=\sum_{k=1}^{n}S_{1}(n,k-1)B_{k}+B_{n+1}.  \label{L1-A}
\end{equation}
\end{lemma}

\begin{remark}
By using (\ref{LahLAH}), we have 
\begin{equation*}
x.x^{_{^{(n)}}}=\sum_{k=1}^{n}\left\vert L(n,k)\right\vert x.x_{(k)}.
\end{equation*}
By applying the Volkenborn integral on $\mathbb{Z}_{p}$ to  the both sides
of the above equation, and using (\ref{L1}), we get another  proof of (\ref%
{LL-1a}).
\end{remark}

\begin{theorem}
\begin{equation*}
\int\limits_{\mathbb{Z}_{p}}\frac{x_{(n+1)}}{x}d\mu _{1}\left( x\right)
=(-1)^{n}\sum_{k=0}^{n}n_{(n-k)}\frac{k!}{k+1}.
\end{equation*}
\end{theorem}

\begin{proof}
In order to prove this theorem, we need the following identity: 
\begin{equation}
x_{(n+1)}=x\sum_{k=0}^{n}(-1)^{k}n_{(n-k)}x_{(k)}  \label{IDD-1}
\end{equation}
(\textit{cf}. \cite[p. 58]{Roman}). By applying the Volkenborn integral on $%
\mathbb{Z}_{p}$ to the both sides of the above  equation, and using (\ref%
{Da-0TK}), we arrive at the desired result.
\end{proof}

Since%
\begin{equation*}
\left( x+1\right) _{(n+1)}=x.x_{(n)}+x_{(n)}
\end{equation*}%
and%
\begin{equation*}
\int\limits_{\mathbb{Z}_{p}}\left( x+1\right) _{(n+1)}d\mu _{1}\left(
x\right) =\int\limits_{\mathbb{Z}_{p}}x.x_{(n)}d\mu _{1}\left( x\right)
+\int\limits_{\mathbb{Z}_{p}}x_{(n)}d\mu _{1}\left( x\right) .
\end{equation*}%
Combining the above equation with (\ref{L1}) and (\ref{FF-1}), we get%
\begin{equation*}
\int\limits_{\mathbb{Z}_{p}}\left( x+1\right) _{(n+1)}d\mu _{1}\left(
x\right) =(-1)^{n+1}\frac{n!}{n^{2}+3n+2}+\frac{(-1)^{n}}{n+1}n!.
\end{equation*}%
By using the above equation, we arrive at the following result:

\begin{corollary}
\begin{equation*}
\int\limits_{\mathbb{Z}_{p}}\left( x+1\right) _{(n+1)}d\mu _{1}\left(
x\right) =\frac{(-1)^{n}}{n+2}n!.
\end{equation*}
\end{corollary}

A recurrence relation for the numbers $Y_{1}(n:B)$ is given by the following
theorem:

\begin{theorem}
\begin{equation*}
Y_{1}(n+1:B)+nY_{1}(n:B)=\sum_{k=1}^{n}S_{1}(n,k-1)B_{k}+B_{n+1}
\end{equation*}
or 
\begin{equation}
Y_{1}(n+1:B)+nY_{1}(n:B)=(-1)^{n+1}\frac{n!}{n^{2}+3n+2}.  \label{L2}
\end{equation}
\end{theorem}

\begin{proof}
Since 
\begin{equation*}
x.x_{(n)}=x_{(n+1)}+nx_{(n)}.
\end{equation*}
By applying the Volkenborn integral on $\mathbb{Z}_{p}$ to  the both sides
of the above equation, and using (\ref{Y1}), (\ref{L1}) and ( \ref{L1-A}),
we arrive at the desired result.
\end{proof}

\begin{theorem}
\begin{equation*}
\int\limits_{\mathbb{Z}_{p}}Y_{1}(y,n:B(y))d\mu _{1}\left( y\right)
=(-1)^{n}\sum\limits_{k=0}^{n}\frac{n!}{(k+1)(n-k+1)}.
\end{equation*}
\end{theorem}

\begin{proof}
The well-known Chu-Vandermonde identity is defined as follows 
\begin{equation}
\sum\limits_{k=0}^{n}\left( 
\begin{array}{c}
x \\ 
k%
\end{array}
\right) \left( 
\begin{array}{c}
y \\ 
n-k%
\end{array}
\right) =\left( 
\begin{array}{c}
x+y \\ 
n%
\end{array}
\right) .  \label{cv}
\end{equation}
By applying the Volkenborn integral on $\mathbb{Z}_{p}$ wrt  $x$ and $y$ to
the LHS of Equation (\ref{cv}), and using (\ref{C7}), we get 
\begin{equation}
LHS=(-1)^{n}\sum\limits_{k=0}^{n}\frac{1}{(k+1)(n-k+1)}.  \label{k1}
\end{equation}
By applying the Volkenborn integral on $\mathbb{Z}_{p}$ wrt  $x$ and $y$ to
the RHS of Equation (\ref{cv}), and using (\ref{Y1a}) and ( \ref{Y2a}), we
obtain 
\begin{equation}
RHS=\frac{1}{n!}\int\limits_{\mathbb{Z}_{p}}Y_{1}(y,n:B(y))d\mu _{1}\left(
y\right) .  \label{k2}
\end{equation}
Combining (\ref{k1}) with (\ref{k2}), we arrived at the desired result.
\end{proof}

Since%
\begin{equation}
\left( 
\begin{array}{c}
x+y \\ 
n%
\end{array}%
\right) =\frac{1}{n!}\left( x+y\right) _{(n)}=\frac{1}{n!}%
\sum\limits_{k=0}^{n}\left( 
\begin{array}{c}
n \\ 
k%
\end{array}%
\right) x_{(k)}y_{(n-k)}  \label{LamdaFun-1d}
\end{equation}%
By applying the Volkenborn integral on $\mathbb{Z}_{p}$ wrt $x$ and $y$ to
equation (\ref{cv}), and using (\ref{Da-0TK}) and (\ref{Da-0TKa}), we also
get the following lemma:

\begin{lemma}
\begin{equation}
\int\limits_{\mathbb{Z}_{p}}\int\limits_{\mathbb{Z}_{p}}\left( 
\begin{array}{c}
x+y \\ 
n%
\end{array}
\right) d\mu _{1}\left( y\right) d\mu _{1}\left( y\right)
=\sum\limits_{k=0}^{n}(-1)^{n}\frac{1}{(k+1)(n-k+1)}.  \label{LamdaFun-1a}
\end{equation}
\end{lemma}

Since%
\begin{equation*}
\left( 
\begin{array}{c}
x+y \\ 
n%
\end{array}%
\right) =\frac{1}{n!}\left( x+y\right) _{(n)}=\frac{1}{n!}%
\sum\limits_{k=0}^{n}S_{1}(n,k)\left( x+y\right) ^{k}
\end{equation*}%
By applying the Volkenborn integral on $\mathbb{Z}_{p}$ wrt $x$ and $y$ to
equation (\ref{cv}), and using (\ref{Da-0TK}) and (\ref{Da-0TKa}), we also
get the following lemma:

\begin{lemma}
\begin{equation}
\int\limits_{\mathbb{Z}_{p}}\int\limits_{\mathbb{Z}_{p}}\left( 
\begin{array}{c}
x+y \\ 
n%
\end{array}
\right) d\mu _{1}\left( y\right) d\mu _{1}\left( y\right) =\frac{1}{n!}
\sum\limits_{k=0}^{n}\sum\limits_{j=0}^{k}\left( 
\begin{array}{c}
k \\ 
j%
\end{array}
\right) S_{1}(n,k)B_{j}B_{k-j}.  \label{LamdaFun-1b}
\end{equation}
\end{lemma}

\begin{lemma}
\begin{equation}
\int\limits_{\mathbb{Z}_{p}}\left( 
\begin{array}{c}
x+1 \\ 
n%
\end{array}
\right) d\mu _{1}\left( x\right) =\frac{(-1)^{n+1}}{n^{2}+n}.  \label{v1a}
\end{equation}
\end{lemma}

\begin{proof}
In \cite{Schikof} the following integral formula is given 
\begin{equation}
\int\limits_{\mathbb{Z}_{p}}f(x+n)d\mu _{1}\left( x\right) =\int\limits_{ 
\mathbb{Z}_{p}}f(x)d\mu _{1}\left( x\right)
+\sum\limits_{k=0}^{n-1}f^{\prime }(k),  \label{v1}
\end{equation}
where 
\begin{equation*}
f^{\prime }(x)=\frac{d}{dx}\left\{ f(x)\right\} .
\end{equation*}
By substituting 
\begin{equation*}
f(x)=\left( 
\begin{array}{c}
x \\ 
n%
\end{array}
\right)
\end{equation*}
into (\ref{v1}), we get 
\begin{equation*}
\int\limits_{\mathbb{Z}_{p}}\left( 
\begin{array}{c}
x+1 \\ 
n%
\end{array}
\right) d\mu _{1}\left( x\right) =\int\limits_{\mathbb{Z}_{p}}\left( 
\begin{array}{c}
x \\ 
n%
\end{array}
\right) d\mu _{1}\left( x\right) +\frac{d}{dx}\left\{ \left( 
\begin{array}{c}
x \\ 
n%
\end{array}
\right) \right\} \left\vert _{x=0}\right.
\end{equation*}
and 
\begin{eqnarray*}
\frac{d}{dx}\left\{ \left( 
\begin{array}{c}
x \\ 
n%
\end{array}
\right) \right\} \left\vert _{x=0}\right. &=&\left\{ \frac{1}{n!}
(x)_{n}\sum\limits_{k=0}^{n-1}\frac{1}{x-k}\right\} \left\vert _{x=0}\right.
\\
&=&(-1)^{n-1}\frac{1}{n}.
\end{eqnarray*}
Therefore 
\begin{equation*}
\int\limits_{\mathbb{Z}_{p}}\left( 
\begin{array}{c}
x+1 \\ 
n%
\end{array}
\right) d\mu _{1}\left( x\right) =\frac{(-1)^{n}}{n+1}+(-1)^{n-1}\frac{1}{n}.
\end{equation*}
After some elementary calculation, we get the desired result.
\end{proof}

\begin{remark}
\begin{equation*}
\Delta \left( 
\begin{array}{c}
x \\ 
n%
\end{array}
\right) =\left( 
\begin{array}{c}
x \\ 
n-1%
\end{array}
\right)
\end{equation*}
and 
\begin{equation*}
\Delta \left( 
\begin{array}{c}
x \\ 
n%
\end{array}
\right) =\left( 
\begin{array}{c}
x+1 \\ 
n%
\end{array}
\right) -\left( 
\begin{array}{c}
x \\ 
n%
\end{array}
\right)
\end{equation*}
Therefore  
\begin{equation}
\left( 
\begin{array}{c}
x+1 \\ 
n%
\end{array}
\right) =\left( 
\begin{array}{c}
x \\ 
n%
\end{array}
\right) +\left( 
\begin{array}{c}
x \\ 
n-1%
\end{array}
\right)  \label{v1-A}
\end{equation}
(cf. \cite[p. 69, Eq-(7)]{Jordan}). By applying the Volkenborn  integral to
the following well-known identities, we also get another proof  of (\ref{v1}%
).
\end{remark}

\begin{corollary}
\begin{equation}
\int\limits_{\mathbb{Z}_{p}}\left( 
\begin{array}{c}
x+1 \\ 
n+1%
\end{array}
\right) d\mu _{1}\left( x\right) =\frac{(-1)^{n}}{n^{2}+3n+2}.  \label{v1b}
\end{equation}
\end{corollary}

\begin{proof}
By applying the Volkenborn integral on $\mathbb{Z}_{p}$ to  the following  
\begin{equation*}
(x+1)_{n+1}=x(x)_{n}+(x)_{n}
\end{equation*}
relation and using (\ref{L1}) and (\ref{YY1a}), we get the desired result.
\end{proof}

\begin{remark}
Replacing $n$ by $n+1$ in (\ref{v1a}), we also get (\ref{v1b}).
\end{remark}

In \cite{Osgood}, Osgood and Wu gave the following identity%
\begin{equation}
(xy)_{(k)}=\sum\limits_{l,m=1}^{k}C_{l,m}^{(k)}(x)_{l}(x)_{m}
\label{LamdaFun-1v}
\end{equation}%
where%
\begin{equation*}
C_{l,m}^{(k)}=\sum\limits_{j=1}^{k}(-1)^{k-j}S_{1}(k,j)S_{2}(j,l)S_{2}(j,m)
\end{equation*}%
$C_{l,m}^{(k)}=C_{m,l}^{(k)}$, $C_{1,1}^{(1)}=1$, $C_{1,1}^{(2)}=0$, $%
C_{1,2}^{(3)}=0=C_{2,1}^{(3)}$. By applying the Volkenborn integral on $%
\mathbb{Z}_{p}$ to equation (\ref{LamdaFun-1v}) wrt $x$ and $y$, we arrive
at the following lemma:

\begin{lemma}
\begin{equation*}
\int\limits_{\mathbb{Z}_{p}}\int\limits_{\mathbb{Z}_{p}}(xy)_{(k)}d\mu
_{1}\left( x\right) d\mu _{1}\left( y\right)
=\sum\limits_{l,m=1}^{k}D_{l}D_{m}C_{l,m}^{(k)}
\end{equation*}
or 
\begin{equation}
\int\limits_{\mathbb{Z}_{p}}\int\limits_{\mathbb{Z}_{p}}(xy)_{(k)}d\mu
_{1}\left( x\right) d\mu _{1}\left( y\right)
=\sum\limits_{l,m=1}^{k}(-1)^{l+m}\frac{l!m!}{(l+1)(m+1)}C_{l,m}^{(k)}.
\label{LamdaFun-1s}
\end{equation}
\end{lemma}

\begin{lemma}
\begin{equation}
\int\limits_{\mathbb{Z}_{p}}\int\limits_{\mathbb{Z}_{p}}(xy)_{(k)}d\mu
_{1}\left( x\right) d\mu _{1}\left( y\right)
=\sum\limits_{m=0}^{k}S_{1}(k,m)\left( B_{l}\right) ^{2}.
\label{LamdaFun-1u}
\end{equation}
\end{lemma}

\begin{proof}
By using (\ref{S2-1a}), we get 
\begin{equation}
(xy)_{(k)}=\sum\limits_{m=0}^{k}S_{1}(k,m)x^{l}y^{l}  \label{LamdaFun-1w}
\end{equation}
By applying the Volkenborn integral on $\mathbb{Z}_{p}$ to  equation (\ref%
{LamdaFun-1w}) wrt $x$ and $y$, and using (\ref{M1}), we get  the desired
result.
\end{proof}

\begin{theorem}
\begin{equation*}
\int\limits_{\mathbb{Z}_{p}}x\left( 
\begin{array}{c}
x-2 \\ 
n-1%
\end{array}
\right) d\mu _{1}\left( x\right) =(-1)^{n}\sum\limits_{k=1}^{n}\frac{k}{k+1}.
\end{equation*}
\end{theorem}

\begin{proof}
Gould \cite[Vol. 3, Eq-(4.20)]{GouldVol3} defined the following identity: 
\begin{equation*}
(-1)^{n}x\left( 
\begin{array}{c}
x-2 \\ 
n-1%
\end{array}
\right) =\sum\limits_{k=1}^{n}(-1)^{k}\left( 
\begin{array}{c}
x \\ 
k%
\end{array}
\right) k.
\end{equation*}
By applying the Volkenborn integral on $\mathbb{Z}_{p}$ to  the above
integral, and using (\ref{C7}), we get the desired result.
\end{proof}

\begin{theorem}
\begin{equation*}
\int\limits_{\mathbb{Z}_{p}}\left( 
\begin{array}{c}
n-x \\ 
n%
\end{array}
\right) d\mu _{-1}\left( x\right) =(-1)^{n}H_{n},
\end{equation*}
where $H_{n}$ denotes the harmonic numbers: 
\begin{equation*}
H_{n}=\sum\limits_{k=0}^{n}\frac{1}{k+1}.
\end{equation*}
\end{theorem}

\begin{proof}
Gould \cite[Vol. 3, Eq-(4.19)]{GouldVol3} defined the following identity: 
\begin{equation*}
(-1)^{n}\left( 
\begin{array}{c}
n-x \\ 
n%
\end{array}
\right) =\sum\limits_{k=0}^{n}(-1)^{k}\left( 
\begin{array}{c}
x \\ 
k%
\end{array}
\right) .
\end{equation*}
By applying the Volkenborn integral on $\mathbb{Z}_{p}$ to  the above
integral, and using (\ref{C7}), we get the desired result.
\end{proof}

\begin{theorem}
\begin{equation*}
\int\limits_{\mathbb{Z}_{p}}\binom{mx}{n}d\mu _{1}\left( x\right)
=\sum_{k=0}^{n}\frac{\left( -1\right) ^{k}}{k+1}\sum_{j=0}^{k}\left(
-1\right) ^{j}\binom{k}{j}\binom{mk-mj}{n}.
\end{equation*}
\end{theorem}

\begin{proof}
Gould \cite[Eq. (2.65)]{GouldV7} gave the following identity 
\begin{equation}
\binom{mx}{n}=\sum_{k=0}^{n}\binom{x}{k}\sum_{j=0}^{k}\left( -1\right) ^{j} 
\binom{k}{j}\binom{mk-mj}{n}.  \label{Id-7}
\end{equation}
By applying the Volkenborn integral to the above equation, and using  (\ref%
{C7}), we arrive at the desired result.
\end{proof}

\begin{theorem}
\begin{equation}
\int\limits_{\mathbb{Z}_{p}}\binom{x}{n}^{r}d\mu _{1}\left( x\right)
=\sum_{k=0}^{nr}\frac{\left( -1\right) ^{k}}{k+1}\sum_{j=0}^{k}\left(
-1\right) ^{j}\binom{k}{j}\binom{k-j}{n}^{r}.  \label{IR-2}
\end{equation}
\end{theorem}

\begin{proof}
Gould \cite[Eq. (2.66)]{GouldV7} gave the following identity 
\begin{equation}
\binom{x}{n}^{r}=\sum_{k=0}^{nr}\binom{x}{k}\sum_{j=0}^{k}\left( -1\right)
^{j}\binom{k}{j}\binom{k-j}{n}^{r}.  \label{Id-5}
\end{equation}
By applying the Volkenborn integral to the above equation, and using  (\ref%
{C7}), we arrive at the desired result.
\end{proof}

\begin{remark}
Substituting $r=1$ into (\ref{IR-2}), since $\binom{k-j}{n}=0$ if $k-j<n$, 
we arrive at equation (\ref{C7}).
\end{remark}

\begin{theorem}
Let $n$ be a positive integer with $n>1$. Then we have 
\begin{equation*}
\int\limits_{\mathbb{Z}_{p}}\left\{ x\binom{x-2}{n-1}+x\left( x-1\right) 
\binom{n-3}{n-2}\right\} d\mu _{1}\left( x\right) =\left( -1\right)
^{n}\sum_{k=0}^{n}\frac{k^{2}}{k+1}.
\end{equation*}
\end{theorem}

\begin{proof}
In \cite[Eq. (2.15)]{GouldV7}, Gould gave the following identity for $n>1$: 
\begin{equation}
x\binom{x-2}{n-1}+x\left( x-1\right) \binom{n-3}{n-2}=\sum_{k=0}^{n}\left(
-1\right) ^{k}\binom{x}{k}k^{2}.  \label{Id-6}
\end{equation}

By applying the Volkenborn integral to the above equation, and using  (\ref%
{C7}), we arrive at the desired result.
\end{proof}

\begin{theorem}
\begin{equation}
\int\limits_{\mathbb{Z}_{p}}\binom{x+n}{n}d\mu _{1}\left( x\right)
=\sum_{k=0}^{n}\frac{\left( -1\right) ^{k}}{k+1}\sum_{j=0}^{k}\left(
-1\right) ^{j}\binom{k}{j}\binom{k-j+n}{n}  \label{Id-1}
\end{equation}
and 
\begin{equation}
\int\limits_{\mathbb{Z}_{p}}\binom{x+n}{n}d\mu _{1}\left( x\right)
=\sum_{k=0}^{n}B_{k}\sum_{j=0}^{n}\binom{n}{j}\frac{S_{1}(j,k)}{j!}.
\label{Id-2}
\end{equation}
\end{theorem}

\begin{proof}
In \cite[Eq. (2.64) and Eq-(6.17)]{GouldV7}, Gould gave the following 
identities: 
\begin{equation}
\binom{x+n}{n}=\sum_{k=0}^{n}\binom{x}{k}\sum_{j=0}^{k}\left( -1\right) ^{j} 
\binom{k}{j}\binom{k-j+n}{n}  \label{Id-1a}
\end{equation}
and 
\begin{equation}
\binom{x+n}{n}=\sum_{k=0}^{n}x^{k}\sum_{j=0}^{n}\binom{n}{j}\frac{S_{1}(j,k) 
}{j!}.  \label{Id-2b}
\end{equation}

By applying the Volkenborn integral to the above equation, and using  (\ref%
{C7}) and (\ref{M1}), respectively, we arrive at the desired result.
\end{proof}

\begin{theorem}
\begin{equation*}
\int\limits_{\mathbb{Z}_{p}}\left( 
\begin{array}{c}
x+n+\frac{1}{2} \\ 
n%
\end{array}
\right) d\mu _{1}\left( x\right) =\left( 
\begin{array}{c}
2n \\ 
n%
\end{array}
\right) \sum\limits_{k=0}^{n}(-1)^{k}\left( 
\begin{array}{c}
n \\ 
k%
\end{array}
\right) \frac{2^{2k-2n}\left( 2n+1\right) }{\left( k+1\right) \left(
2k+1\right) \left( 
\begin{array}{c}
2k \\ 
k%
\end{array}
\right) }.
\end{equation*}
\end{theorem}

\begin{proof}
Gould \cite[Vol. 3, Eq-(6.26)]{GouldVol3} defined the following identity: 
\begin{equation*}
\left( 
\begin{array}{c}
x+n+\frac{1}{2} \\ 
n%
\end{array}
\right) =\left( 2n+1\right) \left( 
\begin{array}{c}
2n \\ 
n%
\end{array}
\right) \sum\limits_{k=0}^{n}\left( 
\begin{array}{c}
n \\ 
k%
\end{array}
\right) \left( 
\begin{array}{c}
x \\ 
k%
\end{array}
\right) \frac{2^{2k-2n}}{\left( 2k+1\right) \left( 
\begin{array}{c}
2k \\ 
k%
\end{array}
\right) }.
\end{equation*}
By applying the Volkenborn integral on $\mathbb{Z}_{p}$ to  the above
integral, and using (\ref{C7}), we get the desired result.
\end{proof}

\begin{theorem}
\begin{equation*}
\int\limits_{\mathbb{Z}_{p}}x^{m}.x_{(n)}d\mu _{1}\left( x\right)
=\sum\limits_{k=0}^{n}S_{1}(n,k)B_{k+m}.
\end{equation*}
\end{theorem}

\begin{proof}
Multiple both sides of equation (\ref{S1a}) by $x^{m}$, we get 
\begin{equation*}
x^{m}.x_{(n)}=\sum\limits_{k=0}^{n}S_{1}(n,k)x^{m+k}
\end{equation*}
By applying the Volkenborn integral on $\mathbb{Z}_{p}$ to  the above
integral, and using (\ref{M1}), we get the desired result.
\end{proof}

In order to give formula for the following integral%
\begin{equation*}
\int\limits_{\mathbb{Z}_{p}}x_{(m)}x_{(n)}d\mu _{1}\left( x\right) ,
\end{equation*}%
we need the following well-known identity%
\begin{equation}
x_{(m)}x_{(n)}=\sum\limits_{k=0}^{m}\left( 
\begin{array}{c}
m \\ 
k%
\end{array}%
\right) \left( 
\begin{array}{c}
n \\ 
k%
\end{array}%
\right) k!x_{(m+n-k)},  \label{LamdaFun-1c}
\end{equation}%
where the coefficients of the $x_{(n+n-k)}$, called connection coefficients,
have a combinatorial interpretation as the number of ways to identify $k$
elements each from a set of size $m$ and a set of size $n$ (\textit{cf}. 
\cite{wikiPEDIAfalling}).

By applying the Volkenborn integral on $\mathbb{Z}_{p}$ to (\ref{LamdaFun-1c}%
) and using (\ref{Y1}), (\ref{L1}) and (\ref{YY1a}), we get the following
lemma:

\begin{lemma}
\begin{equation}
\int\limits_{\mathbb{Z}_{p}}x_{(m)}x_{(n)}d\mu _{1}\left( x\right)
=\sum\limits_{k=0}^{m}(-1)^{m+n-k}\left( 
\begin{array}{c}
m \\ 
k%
\end{array}
\right) \left( 
\begin{array}{c}
n \\ 
k%
\end{array}
\right) Y_{1}(m+n-k:B)  \label{LamdaFun-1g}
\end{equation}
and 
\begin{equation}
\int\limits_{\mathbb{Z}_{p}}x_{(n)}x_{(m)}d\mu _{1}\left( x\right)
=\sum\limits_{k=0}^{m}(-1)^{m+n-k}\left( 
\begin{array}{c}
m \\ 
k%
\end{array}
\right) \left( 
\begin{array}{c}
n \\ 
k%
\end{array}
\right) \frac{k!(m+n-k)!}{m+n-k+1}.  \label{LamdaFun-1f}
\end{equation}
\end{lemma}

\begin{remark}
Since $D_{n}=Y_{1}(n:B)$, we rewrite (\ref{LamdaFun-1g}) as follows: 
\begin{equation*}
\int\limits_{\mathbb{Z}_{p}}x_{(m)}x_{(n)}d\mu _{1}\left( x\right)
=\sum\limits_{k=0}^{m}(-1)^{m+n-k}\left( 
\begin{array}{c}
m \\ 
k%
\end{array}
\right) \left( 
\begin{array}{c}
n \\ 
k%
\end{array}
\right) D_{m+n-k}.
\end{equation*}
\end{remark}

By using (\ref{S1a}), we have%
\begin{equation*}
x_{(m)}x_{(n)}=\sum_{j=0}^{n}\sum_{l=0}^{m}S_{1}(n,k)S_{1}(m,l)x^{j+l}.
\end{equation*}%
By applying the Volkenborn integral on $\mathbb{Z}_{p}$ to the above
equation, and using (\ref{M1}), we get the following lemma:

\begin{lemma}
\begin{equation}
\int\limits_{\mathbb{Z}_{p}}x_{(n)}x_{(m)}d\mu _{1}\left( x\right)
=\sum_{j=0}^{n}\sum_{l=0}^{m}S_{1}(n,k)S_{1}(m,l)B_{j+l}.
\label{LamdaFun-1h}
\end{equation}
\end{lemma}

By combining (\ref{LamdaFun-1c}) with (\ref{S1a}), we get%
\begin{equation*}
x_{(m)}x_{(n)}=\sum\limits_{k=0}^{m}\left( 
\begin{array}{c}
m \\ 
k%
\end{array}%
\right) \left( 
\begin{array}{c}
n \\ 
k%
\end{array}%
\right) k!\sum\limits_{l=0}^{m+n-k}S_{1}(m+n-k,l)x^{l}.
\end{equation*}%
By applying the Volkenborn integral on $\mathbb{Z}_{p}$ to the above
equation, we get the following lemma:

\begin{lemma}
\begin{equation}
\int\limits_{\mathbb{Z}_{p}}x_{(n)}x_{(m)}d\mu _{1}\left( x\right)
=\sum\limits_{k=0}^{m}\left( 
\begin{array}{c}
m \\ 
k%
\end{array}
\right) \left( 
\begin{array}{c}
n \\ 
k%
\end{array}
\right) k!\sum\limits_{l=0}^{m+n-k}S_{1}(m+n-k,l)B_{l}.  \label{LamdaFun-1i}
\end{equation}
\end{lemma}

Since%
\begin{equation*}
\frac{x^{(n)}}{n!}=\frac{(-1)^{n}(-x)_{(n)}}{n!}=\left( 
\begin{array}{c}
x+n-1 \\ 
n%
\end{array}%
\right) ,
\end{equation*}%
it follows that%
\begin{equation}
(x+n-1)_{(n)}=\sum\limits_{m=0}^{n+1}(-1)^{m+n}S_{1}(n,m)x^{m}  \label{IDD-2}
\end{equation}%
(\textit{cf}. \cite[p. 164]{Jordan}). By applying the Volkenborn integral on 
$\mathbb{Z}_{p}$ to Equation (\ref{IDD-2}), and using (\ref{M1}), we get the
following identities:%
\begin{equation}
\int\limits_{\mathbb{Z}_{p}}(x+n-1)_{(n)}d\mu _{1}\left( x\right)
=\sum\limits_{m=0}^{n+1}(-1)^{m+n}S_{1}(n,m)B_{m}  \label{IDD-3}
\end{equation}%
and%
\begin{equation*}
\int\limits_{\mathbb{Z}_{p}}(-x)_{(n)}d\mu _{1}\left( x\right)
=\sum\limits_{m=0}^{n+1}(-1)^{m}S_{1}(n,m)B_{m}.
\end{equation*}

By using (\ref{C0}), we also have%
\begin{equation*}
\int\limits_{\mathbb{Z}_{p}}\left( 
\begin{array}{c}
x+n-1 \\ 
n%
\end{array}%
\right) d\mu _{1}\left( x\right) =\sum\limits_{m=0}^{n}(-1)^{m}\left( 
\begin{array}{c}
n-1 \\ 
n-m%
\end{array}%
\right) \frac{1}{m+1}
\end{equation*}%
and%
\begin{equation}
\int\limits_{\mathbb{Z}_{p}}(x+n-1)_{(n)}d\mu _{1}\left( x\right)
=\sum\limits_{m=0}^{n}(-1)^{m}\left( 
\begin{array}{c}
n-1 \\ 
n-m%
\end{array}%
\right) \frac{n!}{m+1}.  \label{IDD-4}
\end{equation}

Thus%
\begin{equation*}
Y_{2}(n:B)=\int\limits_{\mathbb{Z}_{p}}x^{(n)}d\mu _{1}\left( x\right) =%
\frac{1}{n!}\sum\limits_{m=0}^{n}(-1)^{m}\left( 
\begin{array}{c}
n-1 \\ 
n-m%
\end{array}%
\right) \frac{1}{m+1}.
\end{equation*}

A relationship between the numbers $Y_{1}(n:B)$ and $Y_{2}(n:B)$ is given by
the following theorem:

\begin{theorem}
\begin{equation*}
Y_{2}(n:B)=\sum\limits_{m=0}^{n}\left\vert L(n,k)\right\vert Y_{1}(m:B).
\end{equation*}
\end{theorem}

\begin{proof}
By applying the Volkenborn integral on $\mathbb{Z}_{p}$ to ( \ref{Lah}), and
using (\ref{YY1a}) and (\ref{Y2A}), we get the desired  result.
\end{proof}

\section{Application of the $p$-adic fermionic integral to the falling and
rising factorials}

In this section, we give applications of the $p$-adic fermionic integral on $%
\mathbb{Z}_{p}$ to the falling and rising factorials, we derive some
integral formulas including the Euler numbers and polynomials, the Stirling
numbers, the Lah numbers and the combinatorial sums.

By using the same spirit of the Cauchy numbers of the first kind, by
applying the $p$-adic fermionic integral to the rising factorial and the
falling factorial, respectively, we derive various formulas, identities and
relations.

In \cite{simsek2017ascm}, similar to the Cauchy numbers defined aid of the
Riemann integral, we also studied the Euler numbers sequences by using $p$%
-adic fermionic integral. Let $x_{j}\in \mathbb{Z}$ and $j\in \left\{
1,2,\ldots ,n-1\right\} $ with $n>1$.

We define the sequences $(y_{1}(n:E))$, including the Euler numbers as
follows:%
\begin{equation}
y_{1}(n:E)=E_{n}+\sum_{j=1}^{n-1}(-1)^{j}x_{j}E_{n-j},  \label{y11-a}
\end{equation}%
and%
\begin{equation}
y_{2}(n:E)=E_{n}+\sum_{j=1}^{n-1}x_{j}E_{n-j}.  \label{Y11-b}
\end{equation}

The sequences $(y_{1}(n:E))$ and $(y_{2}(n:E))$ can be computed by the first
and second kind Changhee numbers, which are defined by Kim \textit{et al}. 
\cite{DSkim2} as follows:%
\begin{equation}
Ch_{n}=\int\limits_{\mathbb{Z}_{p}}x_{(n)}d\mu _{-1}\left( x\right)
\label{y1}
\end{equation}%
and%
\begin{equation}
\widehat{Ch}_{n}=\int\limits_{\mathbb{Z}_{p}}x^{(n)}d\mu _{-1}\left( x\right)
\label{y2}
\end{equation}%
or%
\begin{equation*}
\widehat{Ch}_{n}=\int\limits_{\mathbb{Z}_{p}}\left( -x\right) _{(n)}d\mu
_{-1}\left( x\right) .
\end{equation*}%
Combining (\ref{y11-a}) with (\ref{y1}) and (\ref{Y11-a}) with (\ref{y2}),
we easily give the following relations for the general terms of the related
sequences:%
\begin{equation*}
y_{1}(n:E)=Ch_{n}
\end{equation*}%
and%
\begin{equation}
y_{2}(n:E)=\widehat{Ch}_{n}.  \label{AY-22a}
\end{equation}

By using the above formulas, sew values of the sequences $\left(
y_{1}(n:E)\right) $ and $\left( y_{2}(n:E)\right) $ are computed,
respectively, as follows:%
\begin{eqnarray*}
y_{1}(0 &:&E)=E_{0} \\
y_{1}(1 &:&E)=E_{1} \\
y_{1}(2 &:&E)=E_{2}-E_{1} \\
y_{1}(3 &:&E)=E_{3}-3E_{2}+2E_{1} \\
y_{1}(4 &:&E)=E_{4}-6E_{3}+11E_{2}-6E_{1},
\end{eqnarray*}%
and%
\begin{eqnarray*}
y_{2}(0 &:&E)=E_{0} \\
y_{2}(1 &:&E)=E_{1} \\
y_{2}(2 &:&E)=E_{2}+E_{1} \\
y_{2}(3 &:&E)=E_{3}+3E_{2}+2E_{1} \\
y_{2}(4 &:&E)=E_{4}+6E_{3}+11E_{2}+6E_{1}
\end{eqnarray*}%
By using definition of Euler numbers, we have%
\begin{equation*}
y_{1}(0:E)=1,y_{1}(1:E)=-\frac{1}{2},y_{1}(2:E)=\frac{1}{2},y_{1}(3:E)=-%
\frac{3}{4},y_{1}(4:E)=\frac{3}{2},\ldots
\end{equation*}%
and%
\begin{equation*}
y_{2}(0:E)=1,y_{2}(1:E)=-\frac{1}{2},y_{2}(2:E)=-\frac{1}{2},y_{2}(3:E)=-%
\frac{3}{4},y_{2}(4:E)=-\frac{3}{2},\ldots
\end{equation*}

Kim \textit{et al}. \cite{DSkim2} defined the first and second kind Changhee
polynomials, respectively, as follows:%
\begin{equation}
Ch_{n}(x)=\int\limits_{\mathbb{Z}_{p}}\left( x+t\right) _{(n)}d\mu
_{-1}\left( t\right)  \label{y1a}
\end{equation}%
or%
\begin{equation}
\widehat{Ch}_{n}(x)=\int\limits_{\mathbb{Z}_{p}}\left( x+t\right) ^{(n)}d\mu
_{-1}\left( t\right) .  \label{y2a}
\end{equation}

We also define the following sequences for the Euler polynomials, $%
(y_{1}(x,n:E(x)))$ and $(y_{2}(x,n:E(x)))$\ related to the polynomials $%
Ch_{n}(x)$\textit{\ }and $\widehat{Ch_{n}}(x)$ as follows:%
\begin{equation*}
y_{1}(x,n:E(x))=E_{n}(x)+\sum_{j=1}^{n-1}(-1)^{j}x_{j}E_{n-j}(x),
\end{equation*}%
and%
\begin{equation*}
y_{2}(x,n:E(x))=E_{n}(x)+\sum_{j=1}^{n-1}x_{j}E_{n-j}(x).
\end{equation*}

Observe that when $x=0$, the sequences $(y_{1}(x,n:E(x)))$ and $%
(y_{2}(x,n:E(x)))$ reduces to the following sequences:%
\begin{equation*}
y_{1}(n:E)=y_{1}(0,n:E(0))
\end{equation*}%
and%
\begin{equation*}
y_{2}(n:E)=y_{2}(0,n:E(0)).
\end{equation*}

\section{Formulas for the sequence $y_{1}(n:E)$}

Using the fermionic integral and its integral equations, we give some
formula identities of the sequence $\left( y_{1}(n:E)\right) $. We also
gives some $p$-adic fermionic integral formulas including the falling
factorials.

Explicit formula for the sequence $y_{1}(n:E)$ is given by the following
theorem, which was proved by different method (cf. \cite{DSkimDaehee}, \cite%
{DSkim2} \cite{simsekCogent}, \cite{simsek2017ascm}).

\begin{theorem}
\begin{equation*}
y_{1}(n:E)=(-1)^{n}2^{-n}n!.
\end{equation*}
\end{theorem}

\begin{proof}
We know that numbers of the sequence $\left( y_{1}(n:E)\right) $ are related
ot the numbers $Ch_{n} $. By using same computaion of the numbers $Ch_{n} $,
this theorem is also proved. That is, we now briefly give this proof. Since 
Since 
\begin{equation*}
x_{(n)}=n!\left( 
\begin{array}{c}
x \\ 
j%
\end{array}
\right) ,
\end{equation*}
by using (\ref{est-3}), we get 
\begin{eqnarray}
\int\limits_{\mathbb{Z}_{p}}x_{(n)}d\mu _{-1}\left( x\right)
&=&n!\int\limits_{\mathbb{Z}_{p}}\left( 
\begin{array}{c}
x \\ 
n%
\end{array}
\right) d\mu _{-1}\left( x\right)  \label{FF-1A} \\
&=&\frac{(-1)^{n}}{2^{n}}n!.  \notag
\end{eqnarray}
Thus, we get the desired result.
\end{proof}

\begin{theorem}
\begin{equation}
\int\limits_{\mathbb{Z}_{p}}\left( x+1\right) _{(n)}d\mu _{-1}\left(
x\right) =(-1)^{n+1}\frac{1}{2^{n}}n!.  \label{v1-B}
\end{equation}
\end{theorem}

\begin{proof}
Since 
\begin{equation*}
\Delta x_{(n)}=\left( x+1\right) _{(n)}-x_{(n)},
\end{equation*}
we have 
\begin{equation*}
\left( x+1\right) _{(n)}=x_{(n)}+nx_{(n-1)}.
\end{equation*}
By applying the fermionic integral on $\mathbb{Z}_{p}$ to the above equation
and combining with (\ref{FF-1A}), we get the desired result.
\end{proof}

\begin{theorem}
\begin{equation*}
\int\limits_{\mathbb{Z}_{p}}\frac{x_{(n+1)}}{x}d\mu _{-1}\left( x\right)
=(-1)^{n}\sum_{k=0}^{n}n_{(n-k)}\frac{k!}{2^{k}}.
\end{equation*}
\end{theorem}

\begin{proof}
By applying the Volkenborn integral on $\mathbb{Z}_{p}$ to  Equation (\ref%
{IDD-1}), and using (\ref{FF-1A}), we arrive at the desired  result.
\end{proof}

\begin{remark}
By applying the fermionic Volkenborn integral with (\ref{est-3}) to Equation
(\ref{v1-A}), we get 
\begin{equation*}
\int\limits_{\mathbb{Z}_{p}}\binom{x+1}{n}d\mu _{-1}\left( x\right)
=(-1)^{n+1}\frac{1}{2^{n}}.
\end{equation*}
By using the above equation, we also get another proof of (\ref{v1-B}).
\end{remark}

\begin{theorem}
\begin{equation*}
y_{1}(n+1:E)+ny_{1}(n:E)=(-1)^{n}\frac{n!(n-1)}{2^{n+1}}.
\end{equation*}
\end{theorem}

\begin{proof}
\begin{equation*}
x.x_{(n)}=x_{(n+1)}+nx_{(n)}.
\end{equation*}
By applying the fermionic $p$-adic Volkenborn integral on $\mathbb{Z}_{p}$ 
to the both sides of the above equation, and using (\ref{est-3}), (\ref{y1}
), respectively, we get 
\begin{equation*}
\int\limits_{\mathbb{Z}_{p}}x.x_{(n)}d\mu _{1}\left( x\right) =(-1)^{n}\frac{
n!(n-1)}{2^{n+1}},
\end{equation*}
and 
\begin{equation*}
\int\limits_{\mathbb{Z}_{p}}x.x_{(n)}d\mu _{-1}\left( x\right)
=y_{1}(n+1:E)+ny_{1}(n:E).
\end{equation*}
Combining the above equation, we get the desired result.
\end{proof}

By applying the $p$-adic fermionic integral on $\mathbb{Z}_{p}$ to equation (%
\ref{LamdaFun-1v}) wrt $x$ and $y$, we arrive at the following lemma:

\begin{lemma}
\begin{equation}
\int\limits_{\mathbb{Z}_{p}}\int\limits_{\mathbb{Z}_{p}}(xy)_{(k)}d\mu
_{-1}\left( x\right) d\mu _{-1}\left( y\right)
=\sum\limits_{l,m=1}^{k}(-1)^{l+m}2^{-m-l}C_{l,m}^{(k)}.  \label{LamdaFun-1y}
\end{equation}
\end{lemma}

\begin{lemma}
\begin{equation}
\int\limits_{\mathbb{Z}_{p}}\int\limits_{\mathbb{Z}_{p}}(xy)_{(k)}d\mu
_{-1}\left( x\right) d\mu _{-1}\left( y\right)
=\sum\limits_{m=0}^{k}S_{1}(k,m)\left( E_{l}\right) ^{2}.
\label{LamdaFun-1z}
\end{equation}
\end{lemma}

\begin{proof}
By applying the $p$-adic fermionic integral on $\mathbb{Z}_{p}$ to  equation
(\ref{LamdaFun-1w}) wrt $x$ and $y$, and using (\ref{Mm1}), we get  the
desired result.
\end{proof}

\begin{theorem}
Let $n$ be a positive integer with $n>1$. Then we have 
\begin{equation*}
\int\limits_{\mathbb{Z}_{p}}\left\{ x\binom{x-2}{n-1}+x\left( x-1\right) 
\binom{n-3}{n-2}\right\} d\mu _{-1}\left( x\right) =\left( -1\right)
^{n}\sum_{k=0}^{n}\frac{k^{2}}{2^{k}}.
\end{equation*}
\end{theorem}

\begin{proof}
By applying the fermionic Volkenborn integral to equation (\ref{Id-6}) with
( \ref{est-3}), we get desired result.
\end{proof}

\begin{theorem}
\begin{equation}
\int\limits_{\mathbb{Z}_{p}}\binom{x+n}{n}d\mu _{-1}\left( x\right)
=\sum_{k=0}^{n}\frac{\left( -1\right) ^{k}}{2^{k}}\sum_{j=0}^{k}\left(
-1\right) ^{j}\binom{k}{j}\binom{k-j+n}{n}  \label{Id-3}
\end{equation}
and 
\begin{equation}
\int\limits_{\mathbb{Z}_{p}}\binom{x+n}{n}d\mu _{-1}\left( x\right)
=\sum_{k=0}^{n}E_{k}\sum_{j=0}^{n}\binom{n}{j}\frac{S_{1}(j,k)}{j!}.
\label{Id-4}
\end{equation}
\end{theorem}

\begin{proof}
By applying the fermionic Volkenborn integral to Equations (\ref{Id-1a})
and( \ref{Id-2b}) with (\ref{est-3}) and (\ref{Mm1}), we get desired result.
\end{proof}

\begin{theorem}
\begin{equation*}
\int\limits_{\mathbb{Z}_{p}}\binom{mx}{n}d\mu _{-1}\left( x\right)
=\sum_{k=0}^{n}\frac{\left( -1\right) ^{k}}{2^{k}}\sum_{j=0}^{k}\left(
-1\right) ^{j}\binom{k}{j}\binom{mk-mj}{n}.
\end{equation*}
\end{theorem}

\begin{proof}
By applying the fermionic Volkenborn integral to Equation (\ref{Id-7}) with
( \ref{est-3}), we get desired result.
\end{proof}

\begin{theorem}
\begin{equation}
\int\limits_{\mathbb{Z}_{p}}\binom{x}{n}^{r}d\mu _{-1}\left( x\right)
=\sum_{k=0}^{nr}\frac{\left( -1\right) ^{k}}{2^{k}}\sum_{j=0}^{k}\left(
-1\right) ^{j}\binom{k}{j}\binom{k-j}{n}^{r}.  \label{IR-1}
\end{equation}
\end{theorem}

\begin{proof}
By applying the fermionic Volkenborn integral to Equation (\ref{Id-5}) with
( \ref{est-3}), we get desired result.
\end{proof}

\begin{remark}
Substituting $r=1$ into (\ref{IR-1}), since $\binom{k-j}{n}=0$ if $k-j<n$, 
we arrive at equation (\ref{est-3}).
\end{remark}

\begin{theorem}
\begin{equation*}
\int\limits_{\mathbb{Z}_{p}}x\left( 
\begin{array}{c}
x-2 \\ 
n-1%
\end{array}
\right) d\mu _{-1}\left( x\right) =(-1)^{n}\sum\limits_{k=1}^{n}k2^{-k}.
\end{equation*}
\end{theorem}

\begin{proof}
Gould \cite[Vol. 3, Eq-(4.20)]{GouldVol3} defined the following identity: 
\begin{equation*}
(-1)^{n}x\left( 
\begin{array}{c}
x-2 \\ 
n-1%
\end{array}
\right) =\sum\limits_{k=1}^{n}(-1)^{k}\left( 
\begin{array}{c}
x \\ 
k%
\end{array}
\right) k.
\end{equation*}
By applying the fermionic $p$-adic integral on $\mathbb{Z}_{p}$ to the above
integral, and using (\ref{est-3}), we get the desired result.
\end{proof}

\begin{theorem}
\begin{equation*}
\int\limits_{\mathbb{Z}_{p}}\left( 
\begin{array}{c}
n-x \\ 
n%
\end{array}
\right) d\mu _{-1}\left( x\right) =(-1)^{n}\sum\limits_{k=1}^{n}2^{-k}.
\end{equation*}
\end{theorem}

\begin{proof}
Gould \cite[Vol. 3, Eq-(4.19)]{GouldVol3} defined the following identity: 
\begin{equation*}
(-1)^{n}\left( 
\begin{array}{c}
n-x \\ 
n%
\end{array}
\right) =\sum\limits_{k=1}^{n}(-1)^{k}\left( 
\begin{array}{c}
x \\ 
k%
\end{array}
\right) .
\end{equation*}
By applying the fermionic $p$-adic integral on $\mathbb{Z}_{p}$ to the above
integral, and using (\ref{est-3}), we get the desired result.
\end{proof}

\begin{theorem}
\begin{equation*}
\int\limits_{\mathbb{Z}_{p}}\left( 
\begin{array}{c}
x+n+\frac{1}{2} \\ 
n%
\end{array}
\right) d\mu _{-1}\left( x\right) =\left( 2n+1\right) \left( 
\begin{array}{c}
2n \\ 
n%
\end{array}
\right) \sum\limits_{k=0}^{n}(-1)^{k}\left( 
\begin{array}{c}
n \\ 
k%
\end{array}
\right) \frac{2^{k-2n}}{\left( 2k+1\right) \left( 
\begin{array}{c}
2k \\ 
k%
\end{array}
\right) }.
\end{equation*}
\end{theorem}

\begin{proof}
Gould \cite[Vol. 3, Eq-(6.26)]{GouldVol3} defined the following identity: 
\begin{equation*}
\left( 
\begin{array}{c}
x+n+\frac{1}{2} \\ 
n%
\end{array}
\right) =\left( 2n+1\right) \left( 
\begin{array}{c}
2n \\ 
n%
\end{array}
\right) \sum\limits_{k=0}^{n}\left( 
\begin{array}{c}
n \\ 
k%
\end{array}
\right) \left( 
\begin{array}{c}
x \\ 
k%
\end{array}
\right) \frac{2^{2k-2n}}{\left( 2k+1\right) \left( 
\begin{array}{c}
2k \\ 
k%
\end{array}
\right) }.
\end{equation*}
By applying the fermionic $p$-adic integral on $\mathbb{Z}_{p}$ to the above
integral, and using (\ref{est-3}), we get the desired result.
\end{proof}

\section{Formulas for the sequence $y_{2}(n:E)$}

By using the fermionic integral and its integral equations, we derive some
formula identities of the sequence $\left( y_{2}(n:E)\right) $. We also
gives some $p$-adic fermionic integral formulas including the raising
factorials.

\begin{theorem}
\begin{equation*}
y_{2}(n:E)=\sum_{k=1}^{n}(-1)^{k}\left\vert L(n,k)\right\vert k!2^{-k}.
\end{equation*}
\end{theorem}

\begin{proof}
By applying the $p$-adic fermionic integral on $\mathbb{Z}_{p}$ to the 
equation (\ref{LahLAH}), we get 
\begin{equation*}
y_{2}(n:E)=\sum_{k=1}^{n}\left\vert L(n,k)\right\vert \int\limits_{\mathbb{Z}
_{p}}x_{(k)}d\mu _{-1}\left( x\right)
\end{equation*}
By substituting (\ref{est-3}), we get the desired result.
\end{proof}

By applying the $p$-adic fermionic integral on $\mathbb{Z}_{p}$ to equation (%
\ref{AY-2}) and using (\ref{Mm1}), we get the following theorem, which
modify equation (\ref{AY-22a}):

\begin{theorem}
\begin{equation*}
y_{2}(n,E)=\sum_{k=1}^{n}C(n,k)E_{k}
\end{equation*}
\end{theorem}

By applying the $p$-adic fermionic integral on $\mathbb{Z}_{p}$ to (\ref%
{LahLAH}), and using (\ref{est-3}), we get%
\begin{equation*}
y_{2}(n:E)=\sum_{k=0}^{n}\left\vert L(n,k)\right\vert k!\int\limits_{\mathbb{%
\ Z}_{p}}\left( 
\begin{array}{c}
x \\ 
k%
\end{array}%
\right) d\mu _{-1}\left( x\right) .
\end{equation*}
After some elementary calculations, we arrive at the following formula for
the numbers $y_{2}(n:E)$:

\begin{theorem}
\begin{equation*}
y_{2}(n:E)=\sum_{k=0}^{n}(-1)^{k}\left\vert L(n,k)\right\vert k!2^{-k}.
\end{equation*}
\end{theorem}

By applying the $p$-adic fermionic integral on $\mathbb{Z}_{p}$ to equation (%
\ref{IR-3}), and using (\ref{Mm1}), we arrive at the following formula for
the numbers $y_{2}(n:E)$:

\begin{theorem}
\begin{equation*}
y_{2}(n:B)=\sum_{k=0}^{n}\sum_{j=0}^{k}\left\vert L(n,k)\right\vert
S_{1}(k,j)E_{j}.
\end{equation*}
\end{theorem}

In \cite{AM2014}, we gave%
\begin{equation}
\int\limits_{\mathbb{Z}_{p}}\left( x+n-1\right) _{(n)}d\mu _{-1}\left(
x\right) =n!\sum_{m=0}^{n}(-1)^{m}\left( 
\begin{array}{c}
n-1 \\ 
n-m%
\end{array}%
\right) 2^{-m}.  \label{IDD-6}
\end{equation}%
Since%
\begin{equation*}
x^{(n)}=(-1)^{n}\left( -x\right) _{(n)},
\end{equation*}%
using the $p$-adic fermionic integral with (\ref{IDD-6}), we have%
\begin{eqnarray*}
\int\limits_{\mathbb{Z}_{p}}x^{(n)}d\mu _{-1}\left( x\right)
&=&(-1)^{n}\int\limits_{\mathbb{Z}_{p}}\left( -x\right) _{(n)}d\mu
_{-1}\left( x\right) \\
&=&n!\sum_{m=0}^{n}(-1)^{m}\left( 
\begin{array}{c}
n-1 \\ 
n-m%
\end{array}
\right) 2^{-m}.
\end{eqnarray*}%
From the above equation, we arrive at the following theorem:

\begin{theorem}
\begin{equation*}
y_{2}(n:E)=n!\sum_{m=0}^{n}(-1)^{m}\left( 
\begin{array}{c}
n-1 \\ 
n-m%
\end{array}
\right) 2^{-m}
\end{equation*}
\end{theorem}

By applying the fermionic $p$-adic integral on $\mathbb{Z}_{p}$ to equation (%
\ref{IDD-2}), and using (\ref{Mm1}), we get the following identities:%
\begin{equation}
\int\limits_{\mathbb{Z}_{p}}(x+n-1)_{(n)}d\mu _{-1}\left( x\right)
=\sum\limits_{m=0}^{n+1}(-1)^{m+n}S_{1}(n,m)E_{m}.  \label{IDD-5}
\end{equation}%
By combining the above equation with (\ref{IDD-6}), we get the following
formulas for the sequence $y_{2}(n:E)$:

\begin{theorem}
\begin{equation*}
y_{2}(n:E)=\sum\limits_{m=0}^{n+1}(-1)^{m+n}S_{1}(n,m)E_{m}.
\end{equation*}
\end{theorem}

\section{Identities for combinatorial sums including special numbers}

In this section, by using integral formulas, we derive many novel the
combinatorial sums including the Bernoulli numbers, the Euler numbers, the
Stirling numbers, the Eulerian numbers and the Lah numbers.

Combining (\ref{LamdaFun-1f}) and (\ref{LamdaFun-1h}), we arrive at the
following theorem:

\begin{theorem}
\begin{equation}
\sum_{j=0}^{n}\sum_{l=0}^{m}S_{1}(n,k)S_{1}(m,l)B_{j+l}=\sum
\limits_{k=0}^{m}(-1)^{m+n-k}\left( 
\begin{array}{c}
m \\ 
k%
\end{array}
\right) \left( 
\begin{array}{c}
n \\ 
k%
\end{array}
\right) \frac{k!(m+n-k)!}{m+n-k+1}.  \label{LamdaFun-1j}
\end{equation}
\end{theorem}

Combining (\ref{LamdaFun-1h}) and (\ref{LamdaFun-1i}), we arrive at the
following theorem:

\begin{theorem}
\begin{equation}
\sum_{j=0}^{n}\sum_{l=0}^{m}S_{1}(n,k)S_{1}(m,l)B_{j+l}=\sum
\limits_{k=0}^{m}\left( 
\begin{array}{c}
m \\ 
k%
\end{array}
\right) \left( 
\begin{array}{c}
n \\ 
k%
\end{array}
\right) k!\sum\limits_{l=0}^{m+n-k}S_{1}(m+n-k,l)B_{l}.  \label{LamdaFun-1k}
\end{equation}
\end{theorem}

By combining (\ref{LamdaFun-1j}) and (\ref{LamdaFun-1k}), we get the
following combinatorial sum by the following the following corollary:

\begin{corollary}
\begin{equation*}
\sum\limits_{k=0}^{m}\left( 
\begin{array}{c}
m \\ 
k%
\end{array}
\right) \left( 
\begin{array}{c}
n \\ 
k%
\end{array}
\right) k!\sum\limits_{l=0}^{m+n-k}S_{1}(m+n-k,l)B_{l}=\sum
\limits_{k=0}^{m}(-1)^{m+n-k}\left( 
\begin{array}{c}
m \\ 
k%
\end{array}
\right) \left( 
\begin{array}{c}
n \\ 
k%
\end{array}
\right) \frac{k!(m+n-k)!}{m+n-k+1}.
\end{equation*}
\end{corollary}

By combining (\ref{LamdaFun-1a}) and (\ref{LamdaFun-1b}), we arrive at the
following theorem:

\begin{theorem}
\begin{equation*}
\sum\limits_{k=0}^{n}\sum\limits_{j=0}^{k}\left( 
\begin{array}{c}
k \\ 
j%
\end{array}
\right) S_{1}(n,k)B_{j}B_{k-j}=\sum\limits_{k=0}^{n}(-1)^{n}\frac{n!}{
(k+1)(n-k+1)}.
\end{equation*}
\end{theorem}

Combining (\ref{L1}) with (\ref{L1-A}), we arrive at the following theorem:

\begin{theorem}
\begin{equation*}
\sum_{k=1}^{n}S_{1}(n,k-1)B_{k}=\frac{(-1)^{n+1}n!}{n^{2}+3n+2}-B_{n+1}.
\end{equation*}
\end{theorem}

Substituting (\ref{Y2B}) into (\ref{LL-1d}), we get we get%
\begin{equation*}
\sum_{k=1}^{n+1}C(n+1,k)B_{k}-n\sum_{k=1}^{n}C(n,k)B_{k}=%
\sum_{k=1}^{n}(-1)^{k+1}\left\vert L(n,k)\right\vert \frac{k!}{k^{2}+3k+2}.
\end{equation*}%
After some elementary calculation in the above equation, we arrive at the
following theorem:

\begin{theorem}
\begin{equation*}
\sum_{k=1}^{n}\left( C(n+1,k)-nC(n,k)\right)
B_{k}=\sum_{k=1}^{n}(-1)^{k+1}\left\vert L(n,k)\right\vert \frac{k!}{
k^{2}+3k+2}-B_{n+1}.
\end{equation*}
\end{theorem}

Combining (\ref{IDD-3}) and (\ref{IDD-4}), we get the following theorem:

\begin{theorem}
Let $n$ be a positive integer. Then we have 
\begin{equation*}
\sum\limits_{m=0}^{n}(-1)^{m}\left( 
\begin{array}{c}
n-1 \\ 
n-m%
\end{array}
\right) \frac{n!}{m+1}=\sum\limits_{m=0}^{n+1}(-1)^{m}S_{1}(n,m)B_{m}.
\end{equation*}
\end{theorem}

Combining (\ref{LL-1a}) and (\ref{LL-1b}) we arrive at the following theorem:

\begin{theorem}
\begin{equation*}
\sum_{k=1}^{n}C(n,k)B_{k+1}=\sum_{k=1}^{n}(-1)^{k+1}\left( 
\begin{array}{c}
n-1 \\ 
k-1%
\end{array}
\right) \frac{n!}{k^{2}+3k+2}.
\end{equation*}
\end{theorem}

By combining (\ref{Id-1}) with (\ref{Id-2}), we get%
\begin{equation*}
\sum_{k=0}^{n}\sum_{j=0}^{n}\binom{n}{j}\frac{S_{1}(j,k)B_{k}}{j!}%
=\sum_{k=0}^{n}\frac{\left( -1\right) ^{k}}{k+1}\sum_{j=0}^{k}\left(
-1\right) ^{j}\binom{k}{j}\binom{k-j+n}{n}.
\end{equation*}%
By substituting (\ref{YY1a}), into the above equation, we arrive at the
following theorem:

\begin{theorem}
\begin{equation*}
\sum_{k=0}^{n}\sum_{j=0}^{n}\binom{n}{j}\frac{S_{1}(j,k)B_{k}}{j!}
=\sum_{k=0}^{n}\sum_{j=0}^{k}\left( -1\right) ^{j}\binom{k}{j}\binom{k-j+n}{%
n }\frac{Y_{1}(k:B)}{k!}.
\end{equation*}
\end{theorem}

By combining (\ref{LamdaFun-1s}) with (\ref{LamdaFun-1u}), we arrive at the
following theorem:

\begin{theorem}
\begin{equation*}
\sum\limits_{l,m=1}^{k}C_{l,m}^{(k)}D_{l}D_{m}=\sum
\limits_{m=0}^{k}S_{1}(k,m)\left( B_{l}\right) ^{2}
\end{equation*}
or 
\begin{equation*}
\sum\limits_{m=0}^{k}S_{1}(k,m)\left( B_{l}\right)
^{2}=\sum\limits_{l,m=1}^{k}(-1)^{l+m}\frac{l!m!}{(l+1)(m+1)}C_{l,m}^{(k)}.
\end{equation*}
\end{theorem}

By combining (\ref{LamdaFun-1y}) with (\ref{LamdaFun-1z}), we arrive at the
following theorem:

\begin{theorem}
\begin{equation*}
\sum\limits_{l,m=1}^{k}(-1)^{l+m}2^{-m-l}C_{l,m}^{(k)}=\sum
\limits_{m=0}^{k}S_{1}(k,m)\left( E_{l}\right) ^{2}.
\end{equation*}
\end{theorem}

In \cite[p. 30]{Aigner}, Aigner gave the following valuable comments on the
Eulerian number $A_{n,k}$:

Let $\sigma =a_{1}a_{2}...a_{n}\in S(n)$, the set of all permutations of $%
\left\{ 1,2,...,n\right\} $, be given in word form. A run in $\sigma $ is a
largest increasing subsequence of consecutive entries. The Eulerian number $%
A_{n,k}$ is the number of $\sigma \in S(n)$ with precisely $k$ runs or
equivalently with $k-1$ descents $a_{i}>a_{i+1}$. Consequently,%
\begin{equation*}
A_{n,1}=A_{n,n}=1
\end{equation*}%
with $12...n$ respectively $nn-1...1$ as the only permutations. Aigner gave
the following a recurrence relation:%
\begin{equation*}
A_{n,k}=(n-k+1)A_{n-1,k-1}+kA_{n-1,k}
\end{equation*}%
for $n,k\geq 1$ with $A_{0,0}=1$, $A_{0,k}=0$ ($k>0$). Explicit formula for
the numbers $A_{n,k}$ is given by%
\begin{equation*}
A_{n,k}=\sum_{j=0}^{k}(-1)^{j}\left( 
\begin{array}{c}
1+n \\ 
j%
\end{array}%
\right) \left( k-j\right) ^{n}.
\end{equation*}

By using the following identity%
\begin{equation*}
x^{n}=\sum_{k=0}^{n}A_{n,k}\left( 
\begin{array}{c}
x+n-k \\ 
n%
\end{array}%
\right)
\end{equation*}%
(\textit{cf}. \cite[p. 30]{Aigner}), we have%
\begin{equation*}
x^{n}=\frac{1}{n!}\sum_{k=0}^{n}A_{n,k}\sum_{j=0}^{n}S_{1}(n,j)%
\sum_{l=0}^{j}\left( 
\begin{array}{c}
j \\ 
j%
\end{array}%
\right) \left( n-k\right) ^{j-l}x^{l}
\end{equation*}

By applying the Volkenborn integral and the $p$-adic fermionic integral to
the above equation, respectively, we get relations between the Bernoulli,
Euler, Eulerian and Stirling numbers by the following theorem:

\begin{theorem}
\begin{equation*}
B_{n}=\frac{1}{n!}\sum_{k=0}^{n}A_{n,k}\sum_{j=0}^{n}S_{1}(n,j)
\sum_{l=0}^{j}\left( 
\begin{array}{c}
j \\ 
j%
\end{array}
\right) \left( n-k\right) ^{j-l}B_{l}
\end{equation*}
and 
\begin{equation*}
E_{n}=\frac{1}{n!}\sum_{k=0}^{n}A_{n,k}\sum_{j=0}^{n}S_{1}(n,j)
\sum_{l=0}^{j}\left( 
\begin{array}{c}
j \\ 
j%
\end{array}
\right) \left( n-k\right) ^{j-l}E_{l}.
\end{equation*}
\end{theorem}

By combining (\ref{IDD-5}) with (\ref{IDD-6}), we obtain the following
theorem:

\begin{theorem}
Let $n$ be a positive integer. Then we have 
\begin{equation*}
\sum_{m=0}^{n}(-1)^{m}\left( 
\begin{array}{c}
n-1 \\ 
n-m%
\end{array}
\right) \frac{n!}{2^{m}}=\sum\limits_{m=0}^{n+1}(-1)^{m+n}S_{1}(n,m)E_{m}.
\end{equation*}
\end{theorem}

By combining left-hand side of (\ref{Id-1}) and (\ref{Id-2}), we get the
following theorem:

\begin{theorem}
\begin{equation*}
\sum_{k=0}^{n}\sum_{j=0}^{n}\binom{n}{j}\frac{S_{1}(j,k)B_{k}}{j!}
=\sum_{k=0}^{n}\sum_{j=0}^{k}\left( -1\right) ^{k+j}\binom{k}{j}\binom{k-j+n 
}{n}\frac{1}{k+1}.
\end{equation*}
\end{theorem}

\begin{theorem}
\begin{equation*}
B_{n}=\sum_{j=0}^{n}\frac{j!}{n!}\sum_{m=0}^{j}\sum_{k=0}^{j}(-1)^{j+k+m}
\left( 
\begin{array}{c}
j-1 \\ 
j-m%
\end{array}
\right) \left( 
\begin{array}{c}
n+1 \\ 
j-k%
\end{array}
\right) \frac{k^{n}}{m+1}.
\end{equation*}
\end{theorem}

\begin{proof}
By applying the Volkenborn integral to the following identity which  was
given by Gould \cite[Eq-(4.1)]{GouldV7} 
\begin{equation}
x^{n}=\sum_{j=0}^{n}\left( 
\begin{array}{c}
x+j-1 \\ 
n%
\end{array}
\right) \sum_{k=0}^{j}(-1)^{j+k}\left( 
\begin{array}{c}
n+1 \\ 
j-k%
\end{array}
\right) k^{n},  \label{GL-1}
\end{equation}
we get 
\begin{equation*}
\int\limits_{\mathbb{Z}_{p}}x^{n}d\mu _{1}\left( x\right)
=\sum_{j=0}^{n}\sum_{k=0}^{j}(-1)^{j+k}\left( 
\begin{array}{c}
n+1 \\ 
j-k%
\end{array}
\right) k^{n}\int\limits_{\mathbb{Z}_{p}}\left( 
\begin{array}{c}
x+j-1 \\ 
n%
\end{array}
\right) d\mu _{1}\left( x\right) .
\end{equation*}
Combining the above equation with (\ref{1BI}) and (\ref{M1}), we get 
\begin{equation*}
B_{n}=\sum_{j=0}^{n}\sum_{k=0}^{j}(-1)^{j+k}\left( 
\begin{array}{c}
n+1 \\ 
j-k%
\end{array}
\right) \frac{j!}{n!}k^{n}\sum_{m=0}^{j}(-1)^{m}\left( 
\begin{array}{c}
j-1 \\ 
j-m%
\end{array}
\right) \frac{1}{m+1}.
\end{equation*}
Proof of Theorem is completed.
\end{proof}

By combining left-hand side of (\ref{Id-3}) and (\ref{Id-4}), we get the
following theorem:

\begin{theorem}
\begin{equation*}
\sum_{k=0}^{n}\sum_{j=0}^{n}\binom{n}{j}\frac{S_{1}(j,k)E_{k}}{j!}
=\sum_{k=0}^{n}\sum_{j=0}^{k}\left( -1\right) ^{j+k}\binom{k}{j}\binom{k-j+n 
}{n}\frac{1}{2^{k}}.
\end{equation*}
\end{theorem}

\begin{theorem}
\begin{equation*}
E_{n}=\sum_{j=0}^{n}\frac{j!}{n!}\sum_{m=0}^{j}\sum_{k=0}^{j}(-1)^{j+k+m}
\left( 
\begin{array}{c}
j-1 \\ 
j-m%
\end{array}
\right) \left( 
\begin{array}{c}
n+1 \\ 
j-k%
\end{array}
\right) \frac{k^{n}}{2^{m}}.
\end{equation*}
\end{theorem}

\begin{proof}
By applying the fermionic Volkenborn integral to Equation (\ref{GL-1}), we 
obtain 
\begin{equation*}
\int\limits_{\mathbb{Z}_{p}}x^{n}d\mu _{-1}\left( x\right)
=\sum_{j=0}^{n}\sum_{k=0}^{j}(-1)^{j+k}\left( 
\begin{array}{c}
n+1 \\ 
j-k%
\end{array}
\right) k^{n}\int\limits_{\mathbb{Z}_{p}}\left( 
\begin{array}{c}
x+j-1 \\ 
n%
\end{array}
\right) d\mu _{-1}\left( x\right) .
\end{equation*}
Combining the above equation with (\ref{1FI}) and (\ref{Mm1}), we get 
\begin{equation*}
E_{n}=\sum_{j=0}^{n}\sum_{k=0}^{j}(-1)^{j+k}\left( 
\begin{array}{c}
n+1 \\ 
j-k%
\end{array}
\right) \frac{j!}{n!}k^{n}\sum_{m=0}^{j}(-1)^{m}\left( 
\begin{array}{c}
j-1 \\ 
j-m%
\end{array}
\right) \frac{1}{2^{m}}.
\end{equation*}
Proof of Theorem is completed.
\end{proof}

\begin{theorem}
\begin{equation*}
\sum_{m=0}^{n}S_{2}(n,m)L(n,k)=\sum_{m=0}^{n}\left( 
\begin{array}{c}
n \\ 
m%
\end{array}
\right) S_{2}(n-m,k)w_{g}^{(k)}(m).
\end{equation*}
\end{theorem}

\begin{proof}
Substituting $t=e^{t}-1$ into (\ref{La}), and combining with (\ref{SN-1}), 
we get the following functional equation: 
\begin{equation}
F_{L}(e^{t}-1,k)=F_{S}(t,k;1)F_{Fu}(t,k)  \label{LamdaFun-1q}
\end{equation}
where 
\begin{equation}
F_{Fu}(t,k)=\frac{1}{\left( 2-e^{t}\right) ^{k}}=\sum_{n=0}^{\infty
}w_{g}^{(k)}(n)\frac{t^{n}}{n!},  \label{IR-6}
\end{equation}
where the numbers $w_{g}^{(k)}(n)$ denotes the Fubini numbers of order $k$. 
By using (\ref{LamdaFun-1q}), we get 
\begin{equation*}
\sum_{n=0}^{\infty }L(n,k)\frac{\left( e^{t}-1\right) ^{n}}{n!}
=\sum_{n=k}^{\infty }S_{2}(n,k)\frac{t^{n}}{n!}\sum_{n=0}^{\infty
}w_{g}^{(k)}(n)\frac{t^{n}}{n!}.
\end{equation*}
By using the Cauchy product from the above equation, we obtain 
\begin{equation*}
\sum_{n=0}^{\infty }\left( \sum_{m=0}^{n}S_{2}(n,m)\left\vert
L(n,k)\right\vert \right) \frac{t^{n}}{n!}=\sum_{n=0}^{\infty }\left(
\sum_{m=0}^{n}\left( 
\begin{array}{c}
n \\ 
m%
\end{array}
\right) S_{2}(n-m,k)w_{g}^{(k)}(m)\right) \frac{t^{n}}{n!}.
\end{equation*}
Comparing the coefficients of $\frac{t^{n}}{n!}$ on both sides of the above 
equation, we get the desired result.
\end{proof}

\begin{remark}
Substituting $k=1$ into equation (\ref{IR-6}), we arrive at equation (\ref%
{IR-5}).
\end{remark}

\begin{theorem}
\begin{equation*}
\sum_{k=0}^{n}C(n,k)B_{k}=\sum_{k=1}^{n}(-1)^{k}\frac{n!}{k+1}\left( 
\begin{array}{c}
n-1 \\ 
k-1%
\end{array}
\right) .
\end{equation*}
\end{theorem}

\begin{proof}
Combining (\ref{Y2A}) with (\ref{Y2B}), we derived the desired result.
\end{proof}

\begin{theorem}
\begin{equation*}
\sum\limits_{k=0}^{n}n_{(n-k)}\frac{k!}{k^{2}+3k+2}=\frac{(n-1)!}{n+1}.
\end{equation*}
\end{theorem}

\begin{proof}
Roman \cite[p 58]{Roman} gave the following identity 
\begin{equation*}
x_{(n+1)}=\sum\limits_{k=0}^{n}(-1)^{n-k}n_{(n-k)}xx_{(k)}
\end{equation*}
By applying the Volkenborn integral on $\mathbb{Z}_{p}$ to  the above
identity, we have 
\begin{equation*}
\int\limits_{\mathbb{Z}_{p}}x_{(n+1)}d\mu _{1}\left( x\right)
=\sum\limits_{k=0}^{n}(-1)^{n-k}n_{(n-k)}\int\limits_{\mathbb{Z}
_{p}}x.x_{(k)}d\mu _{1}\left( x\right) .
\end{equation*}
By combining (\ref{v1a}) and (\ref{L1}) with the above equation, we get the 
desired result.
\end{proof}

\begin{theorem}
\begin{equation*}
\sum_{j=0}^{n}\sum_{k=0}^{\left[ \frac{n-j}{2}\right] }\left( 
\begin{array}{c}
n \\ 
j%
\end{array}
\right) S_{12}(n-j,k)B_{k+j}=(-1)^{n}\frac{n!}{n+1}
\end{equation*}
or 
\begin{equation*}
\sum_{j=0}^{n}\sum_{k=0}^{\left[ \frac{n-j}{2}\right] }\left( 
\begin{array}{c}
n \\ 
j%
\end{array}
\right) S_{12}(n-j,k)B_{k+j}=\sum_{j=0}^{n}s_{1}(n,j)B_{j},
\end{equation*}
and 
\begin{equation*}
\sum_{j=0}^{n}\sum_{k=0}^{\left[ \frac{n-j}{2}\right] }\left( 
\begin{array}{c}
n \\ 
j%
\end{array}
\right) S_{12}(n-j,k)E_{k+j}=(-1)^{n}\frac{n!}{2^{n}}.
\end{equation*}
\end{theorem}

\begin{proof}
In order to prove the assertions of the theorem, we apply the Volkenborn
integral and $p$-adic  fermionic integral to the following identity (cf. 
\cite[p. 123]{C. A. CharalambidesDISCRETE}): 
\begin{equation}
t_{(n)}=\sum_{j=0}^{n}\sum_{k=0}^{\left[ \frac{n-j}{2}\right] }\left( 
\begin{array}{c}
n \\ 
j%
\end{array}
\right) S_{12}(n-j,k)t^{j+k}.  \label{s12}
\end{equation}
After some elementary evaluation, we get the desired result.
\end{proof}

Integrating both sides of (\ref{s12}) from $0$ to $1$ and using definition
of the Cauchy numbers of the first kind, we arrive at the following
corollary:

\begin{corollary}
\begin{equation*}
b_{n}(0)=\sum_{j=0}^{n}\sum_{k=0}^{\left[ \frac{n-j}{2}\right] }\left( 
\begin{array}{c}
n \\ 
j%
\end{array}
\right) \frac{S_{12}(n-j,k)}{j+k+1}.
\end{equation*}
\end{corollary}

\begin{theorem}
\begin{equation*}
\sum_{j=0}^{n}S_{1}(n,k)B_{k}=\frac{(-1)^{n}n!}{n+1},
\end{equation*}
and 
\begin{equation*}
\sum_{j=0}^{n}S_{1}(n,k)E_{k}=\frac{(-1)^{n}n!}{2^{n}}.
\end{equation*}
\end{theorem}

\begin{proof}
In \cite[Vol. 7, Eq-(5.59)]{GouldV7}, Gould gave the following identity: 
\begin{equation}
\left( 
\begin{array}{c}
x \\ 
n%
\end{array}
\right) =\sum_{j=0}^{n}\frac{S_{1}(n,k)}{n!}x^{k}.  \label{LamdaFun-1o}
\end{equation}
Applying the Volkenborn integral and the $p$-adic fermionic integral to the 
above identity with the Riemann integral from $0$ to $1$, respectively, we 
get the assertions of the theorem.
\end{proof}

Applying the Riemann integral to equation (\ref{LamdaFun-1o}) from $0$ to $1$%
, and using (\ref{LamdaFun-1p}), we arrive at the following theorem:

\begin{theorem}
\begin{equation}
b_{n}(0)=\sum_{j=0}^{n}\frac{1}{k+1}S_{1}(n,k).  \label{IR-4}
\end{equation}
\end{theorem}

\begin{remark}
Equation (\ref{IR-4}) has\ been proved by means of the different methods. 
For example, see \cite{Roman}.
\end{remark}

By using the falling factorial and the Volkenborn integral, we can give
another proof.

\begin{theorem}
Let $n$ be a positive integer. Then we have 
\begin{equation}
B_{n}=\sum\limits_{k=0}^{n-1}(-1)^{k}\frac{k!}{k+1}S_{2}(n,k).
\label{LamdaFun-1n}
\end{equation}
\end{theorem}

\begin{proof}
We modify (\ref{S2-1a}) as follows 
\begin{equation*}
x^{n}=\sum\limits_{k=0}^{n}S_{2}(n,k)\left( xx_{(k-1)}-(k-1)x_{(k-1)}\right)
.
\end{equation*}
By applying the Volkenborn integral on $\mathbb{Z}_{p}$ to  above equation,
and using (\ref{M1}), we get 
\begin{equation}
B_{n}=\sum\limits_{k=0}^{n}S_{2}(n,k)\left( \int\limits_{\mathbb{Z}
_{p}}xx_{(k-1)}d\mu _{1}\left( x\right) -(k-1)\int\limits_{\mathbb{Z}
_{p}}x_{(k-1)}d\mu _{1}\left( x\right) \right) .  \label{LamdaFun-1m}
\end{equation}
By substituting (\ref{FF-1}) and (\ref{L1}) into the above equation, after 
some elementary calculations, we arrive at the desired result.
\end{proof}

\begin{remark}
Equation (\ref{LamdaFun-1n}) has been proved by various different methods ( 
\textit{cf}. \cite{BoyadzhievMathMag}, \cite[p. 117]{C. A.
CharalambidesDISCRETE}, \cite{GYU}, \cite{riARDON}, \cite{Srivastava2011},  
\cite{SrivatavaChoi}).
\end{remark}

By substituting (\ref{FF-1}) into (\ref{LamdaFun-1n}), we get a relation
between the Bernoulli numbers and the numbers $Y_{1}(n:B)$ and the Daehee
numbers by the following corollary:

\begin{corollary}
\begin{equation*}
B_{n}=\sum\limits_{k=0}^{n-1}Y_{1}(k:B)S_{2}(n,k)
\end{equation*}
or 
\begin{equation*}
B_{n}=\sum\limits_{k=0}^{n-1}D_{k}S_{2}(n,k).
\end{equation*}
\end{corollary}

\begin{theorem}
\begin{equation*}
B_{n}=\sum\limits_{k=0}^{n}S_{2}(n,k)\sum_{j=0}^{k-1}S_{1}(k,j)B_{j}+\sum
\limits_{k=0}^{n}S_{2}(n,k)B_{k}.
\end{equation*}
\end{theorem}

\begin{proof}
By combining (\ref{S2-1c}) and (\ref{LamdaFun-A}) with (\ref{LamdaFun-1m}), 
we get 
\begin{equation*}
B_{n}=\sum\limits_{k=0}^{n}S_{2}(n,k)\left(
\sum_{j=0}^{k-1}S_{1}(k-1,k-1)B_{j}+B_{k}-(k-1)\int\limits_{\mathbb{Z}
_{p}}x_{(k-1)}d\mu _{1}\left( x\right) \right) .
\end{equation*}
By substituting (\ref{Da-0TK}) into the above equation, we obtain 
\begin{eqnarray*}
B_{n} &=&\sum\limits_{k=0}^{n}S_{2}(n,k)\left(
\sum_{j=0}^{k-1}S_{1}(k-1,k-1)B_{j}+B_{k}-(k-1)
\sum_{j=0}^{k-1}S_{1}(k-1,j)B_{j}\right) \\
&=&\sum\limits_{k=0}^{n}S_{2}(n,k)\sum_{j=0}^{k-1}\left(
S_{1}(k-1,j-1)-(k-1)S_{1}(k-1,j)\right)
B_{j}+\sum\limits_{k=0}^{n}S_{2}(n,k)B_{k}.
\end{eqnarray*}
By combining the above equation with (\ref{S2-1c}), we arrive at the desired
result.
\end{proof}

\begin{acknowledgement}
The first author was supported by the \textit{Scientific Research Project
Administration of Akdeniz University.}
\end{acknowledgement}


\begin{thebibliography}{99}
\bibitem{Aigner} M. Aigner, \textit{A Course in Enumeration},
Springer-Verlag Berlin, Heidelberg, 2007.

\bibitem{apostol} T. M. Apostol, \textit{On the Lerch zeta function},
Pacific J. Math. \textbf{1} (1951), 161--167.

\bibitem{Amice} Y. Amice, \textit{Integration }$p$\textit{-adique selon A.
Volkenborn (Ed.)}, S\'{e}minaire Delange-Pisot-Poitou, Th\'{e}orie des
nombres \textbf{13}(2) (1971-1972), talk no. G4, pp. G1--G9.

\bibitem{Bayad} A. Bayad, Y. Simsek and H. M. Srivastava, \textit{Some array
type polynomials associated with special numbers and polynomials}, Appl.
Math. Compute. \textbf{244} (2014), 149--157.

\bibitem{BoyadzhievMathMag} K. N. Boyadzhiev, \textit{Close encounters with
the Stirling numbers of the second kind}, Math. Mag. \textbf{85} (2012),
252--266.

\bibitem{Boyadzhiev} K. N. Boyadzhiev, \textit{Binomial transform and the
backward difference}, http://arxiv.org/abs/1410.3014v2.

\bibitem{Belbechir} H. Belbachir and I. E. Bousbaa, \textit{Associated Lah
numbers and }$r$\textit{-Stirling numbers},
https://arxiv.org/pdf/1404.5573.pdf

\bibitem{Byrd} P. F. Byrd, \textit{New relations between Fibonacci and
Bernoulli numbers}, Fibonacci Quarterly \textbf{13} (1975), 111--114.

\bibitem{Gradimir} N. P. Cakic and G. V. Milovanovic, \textit{On generalized
Stirling numbers and polynomials}, Mathematica Balkanica \textbf{18} (2004),
241--248.

\bibitem{Chan} C. H. Chang and C. W. Ha, \textit{A multiplication theorem
for the Lerch zeta function and explicit representations of the Bernoulli
and Euler polynomials}, J. Math. Anal. Appl. \textbf{315} (2006), 758--767.

\bibitem{C. A. CharalambidesDISCRETE} C. A. Charalambides, \textit{\
Combinatorial Methods in Discrete Distributions}, A John Wiley and Sons,
Inc., Publication, 2015.

\bibitem{Charamb} C. A. Charalambides, \textit{Enumerative Combinatorics},
Chapman\&Hall/Crc, Press Company, London, New York, 2002.

\bibitem{Coskun} H. Coskun, \textit{Multiple bracket function, Stirling
number, and Lah number identities}, http://arxiv.org/pdf/1212.6573.pdf

\bibitem{Cigler} J. Cigler, \textit{Fibonacci polynomials and central
factorial numbers}, preprint.

\bibitem{Comtet} L. Comtet, \textit{Advanced Combinatorics: The Art of
Finite and Infinite Expansions}, Reidel, Dordrecht and Boston, 1974
(Translated from the French by J. W. Nienhuys).

\bibitem{Drajovic} G. B. Djordjevic and G. V. Milovanovic, \textit{Special
classes of polynomials}, University of Nis, Faculty of Technology Leskovac,
2014.

\bibitem{Do} Y. Do and D. Lim, \textit{On }$(h,q)$\textit{-Daehee numbers
and polynomials}, Adv. Difference Equ. \textbf{2015}(107) (2015).

\bibitem{ElDosky} B. S. El-Desouky and A. Mustafa, \textit{New results and
matrix representation for Daehee and Bernoulli numbers and polynomials},
https://arxiv.org/pdf/1412.8259.pdf.

\bibitem{ElDosky2} B. S. El-Desouky and R. S. Goma,\textit{\ Multiparameter
poly-Cauchy and poly-Bernoulli numbers and polynomials},
https://arxiv.org/pdf/1410.5300.pdf

\bibitem{Garsia} A. Garsia and J. Remmel, \textit{A combinatorial
interpretation of qderangement and }$q$\textit{-Laguerre numbers}, European
J. Combin. \textbf{1} (1980), 47--59.

\bibitem{Good} I. J. Good, \textit{The number of ordering of n candidates
when ties are permitted}, Fibonacci Quart. \textbf{13} (1975), 11--18.

\bibitem{GouldVol3} H. W. Gould, \textit{Fundamentals of Series},
http://www.math.wvu.edu/\symbol{126}gould/Vol.3.PDF

\bibitem{GouldV7} H. W. Gould, \textit{Combinatorial numbers and associated
identities}, http://www.math.wvu.edu/\symbol{126}gould/Vol.7.PDF

\bibitem{GYU} B. N. Guo and F. Qi, \textit{An explicit formula for Bernoulli
numbers in terms of Stirling numbers of the second kind}, J. Ana. Num.
Theor. \textbf{3}(1) (2015), 27--30

\bibitem{KIMjang} L. C. Jang and T. Kim, \textit{A new approach to }$q$%
\textit{-Euler numbers and polynomials}, J. Concr. Appl. Math. \textbf{6}
(2008), 159--168.

\bibitem{JangGENOCCHI} L. C. Jang, T. Kim, D. H. Lee and D.-W. Park, \textit{%
\ An application of polylogarithms in the analogs of Genocchi numbers},
Notes Number Theory Discrete Math. \textbf{7}(3) (2001), 65--69.

\bibitem{jandY1} L. C. Jang and H. K. Pak, \textit{Non-archimedean
integration associated with }$q$\textit{-Bernoulli numbers}, Proc. Jangjeon
Math. Soc. \textbf{5}(2) (2002), 125--129.

\bibitem{Jolany} H. Jolany, H. Sharifi and R. E. Alikelaye, \textit{Some
results for the Apostol-Genocchi polynomials of higher order}, Bull. Malays.
Math. Sci. Soc. \textbf{36}(2) (2013), 465--479.

\bibitem{Jordan} C. Jordan, \textit{Calculus of Finite Differences}, 2nd ed.
Chelsea Publishing Company, New York, 1950.

\bibitem{TKimTAKAO} D. S. Kim, T. Kim, J.-J. Seo and T. Komatsu, \textit{\
Barnes' multiple Frobenius-Euler and poly-Bernoulli mixed-type polynomials},
Adv. Difference Equ. (2014), \textbf{2014}:92.

\bibitem{DSkim2} D. S. Kim, T. Kim and J. Seo, \textit{A note on Changhee
numbers and polynomials}, Adv. Stud. Theor. Phys. \textbf{7} (2013),
993--1003.

\bibitem{DSkimDaehee} D. S. Kim and T. Kim, \textit{Daehee numbers and
polynomials}, Appl. Math. Sci. (Ruse) \textbf{7}(120) (2013), 5969--5976.

\bibitem{DSKIMopenMath} D. S. Kim and T. Kim, \textit{Some identities of
degenerate special polynomials}, Open Math. \textbf{13} (2015), 380--389.

\bibitem{DSKIMfrob} D. S. Kim and T. Kim, \textit{Some new identities of
Frobenius-Euler numbers and polynomials}, J. Ineq. Appl. (2012), \textbf{2012%
}:307.

\bibitem{DSKIMBoole} D. S. Kim and T. Kim, \textit{A note on Boole
polynomials}, Integral Transforms Spec. Funct. \textbf{25}(8) (2014),
627-633.

\bibitem{Khrennikov} A. Khrennikov, $p$\textit{-adic valued distributions
and their applications to the mathematical physics}, Kluwer, Dordreht, 1994.

\bibitem{MSKIM} M. S. Kim and J. W. Son, \textit{Analytic properties of the }%
$\mathit{q}$\textit{-volkenborn integral on the ring of }$p$\textit{-adic
integers}, Bull. Korean Math. Soc. \textbf{44}(1) (2007), 1--12.

\bibitem{T. Kim} T. Kim, $q$\textit{-Volkenborn integration}, Russ. J. Math.
Phys. \textbf{19} (2002), 288--299.

\bibitem{KimITSFdahee} T. Kim, \textit{An invariant }$p$\textit{-adic
integral associated with Daehee numbers}, Integral Transforms Spec. Funct. 
\textbf{13}(1) (2002), 65--69.

\bibitem{Kim2006TMIC} T. Kim, $q$\textit{-Euler numbers and polynomials
associated with }$p$\textit{-adic }$q$\textit{-integral and basic }$q$%
\textit{-zeta function}, Trend Math. Information Center Math. Sciences 
\textbf{9} (2006), 7--12.

\bibitem{KIMjmaa2017} T. Kim, \textit{On the analogs of Euler numbers and
polynomials associated with }$\mathit{p}$\textit{-adic }$\mathit{q}$\textit{%
\ -integral on }$\mathbb{Z}_{\mathit{p}}$\textit{\ at }$\mathit{q=1}$, J.
Math. Anal. Appl., \textbf{331} (2007), 779--792.

\bibitem{KIMaml2008} T. Kim, \textit{An invariant }$p$\textit{-adic }$q$%
\textit{-integral on }$Z_{p}$, Appl. Math. Letters \textbf{21} (2008),
105--108.

\bibitem{KimCHAReul} T. Kim, $p$\textit{-adic }$l$\textit{-functions and
sums of powers}, http://arxiv.org/pdf/math/0605703v1.pdf

\bibitem{KimGenocchi1} T. Kim, \textit{On the }$q$\textit{-extension of
Euler and Genocchi numbers}, J. Math. Anal. Appl. \textbf{326}(2) (2007),
1458--1465.

\bibitem{KimGenocchi2} T. Kim and S. H. Rim, \textit{Some }$q$\textit{\
-Bernoulli numbers of higher order associated with the }$p$\textit{-adic }$q$%
\textit{-integrals}, Indian J. Pure Appl. Math. \textbf{32}(10) (2001),
1565--1570.

\bibitem{Kim2016Bernoull2} T. Kim, D. S. Kim, D. V. Dolgy and J. J. Seo, 
\textit{Bernoulli polynomials of the second kind and their identities
arising from umbral calculus}, J. Nonlinear Sci. Appl. \textbf{9} (2016),
860--869.

\bibitem{TkimJKMS} T. Kim, S.-H. Rim, Y. Simsek and D. Kim, \textit{On the
analogs of Bernoulli and Euler numbers, related identities and zeta and }$l$%
\textit{-functions}, J. Korean Math. Soc. \textbf{45}(2) (2008), 435--453.

\bibitem{Lim} D. Lim, \textit{On the twisted modified }$q$\textit{-Daehee
numbers and polynomials}, Adv. Stud. Theor. Phys. \textbf{9} (4) (2015),
199--211.

\bibitem{Liu} G. D. Liu and H. M. Srivastava, \textit{Explicit formulas for
the N\"{o}rlund polynomials }$B_{n}^{(x)}$\textit{\ and }$b_{n}^{(x)}$,
Comput. Math. Appl. \textbf{51}(9-10) (2006), 1377--1384.

\bibitem{Luo} Q. M. Luo and H. M. Srivastava, \textit{Some generalizations
of the Apostol-Genocchi polynomials and the Stirling numbers of the second
kind}, Appl. Math. Compute. \textbf{217} (2011), 5702--5728.

\bibitem{Merlini} D. Merlini, R. Sprugnoli and M. C. Verri, \textit{The
Cauchy numbers}, Discrete Math. \textbf{306}(16) (2006), 1906--1920.

\bibitem{Osgood} B. Osgood and W. Wu, \textit{Falling factorials, generating
functions, and Conjoint ranking tables}, J. Integer Seq. \textbf{12} (2009),
1--13, Article 09.7.8.

\bibitem{OzdenAMC2014} H. Ozden and Y. Simsek, \textit{Modification and
unification of the Apostol-type numbers and polynomials and their
applications}, Appl. Math. Compute. \textbf{235} (2014), 338--351.

\bibitem{Ozden2019} H. Ozden, I. N. Cangul and Y. Simsek, \textit{A new
approach to q-Genocchi numbers and their interpolation functions}, Nonlinear
Anal. \textbf{71}(12) (2009), e793--e799.

\bibitem{ODS} H. Ozden, D. Kim and Y. Simsek, \textit{Generating functions
for new family of special numbers and polynomials associated with }$p$%
\textit{-adic }$q$\textit{-integral}, preprint.

\bibitem{Qi} F. Qi, \textit{Explicit formulas for computing Bernoulli
numbers of the second kind and Stirling numbers of the first kind}, Filomat 
\textbf{28}(2) (2014), 319--327.

\bibitem{QiLAH} F. Qi, X. T. Shi and F. F. Liu, \textit{Several identities
involving the falling and rising factorials and the Cauchy, Lah, and
Stirling numbers}, preprint

\bibitem{Park0} J. W. Park, \textit{On a }$q$\textit{-analogue of }$(h,q)$%
\textit{-Daehee numbers and polynomials of higher order}, J. Compute. Analy.
Appl. \textbf{21}(1) (2016), 769--776.

\bibitem{Roman} S. Roman, \textit{The Umbral Calculus}, Dover Publ. Inc.,
New York, 2005.

\bibitem{riARDON} J. Riordan, \textit{Introduction to Combinatorial Analysis}%
, Princeton University Press, 1958.

\bibitem{RyooCHARbernoul} C. S. Ryoo, D. V. Dolgy, H. I. Kwon and Y. S.
Jang, \textit{Functional equations associated with generalized Bernoulli
numbers and polynomials}, Kyungpook Math. J. \textbf{55} (2015), 29--39.

\bibitem{Schikof} W. H. Schikhof, \textit{Ultrametric Calculus: An
Introduction to }$p$\textit{-adic Analysis}, Cambridge Studies in Advanced
Mathematics 4, Cambridge University Press Cambridge, 1984.

\bibitem{jnt2015} Y. Simsek, $\mathit{q}$\textit{-analogue of the twisted }$%
\mathit{l}$\textit{-Series and }$\mathit{q}$\textit{-twisted Euler numbers},
J. Number Theory \textbf{100}(2) (2005), 267--278.

\bibitem{SimsekFPTA} Y. Simsek, \textit{Generating functions for generalized
Stirling type numbers, array type polynomials, Eulerian type polynomials and
their alications}, Fixed Point Theory Appl. \textbf{87} (2013), 343--1355.

\bibitem{YsimManisa} Y. Simsek, \textit{Identities associated with
gnearalized Stirling type numbers and Eulerian polynomials}, Math. Comput.
Appl. \textbf{8}(3) (2013), 251--263.

\bibitem{AM2014} Y. Simsek, \textit{Special numbers on analytic functions},
Applied Math. \textbf{5} (2014), 1091--1098.

\bibitem{simsekP2} Y. Simsek, \textit{Computation methods for combinatorial
sums and Euler type numbers related to new families of numbers}, Math. Meth.
Appl. Sci. (2016), doi: 10.1002/mma.4143.

\bibitem{ysimsek Ascm} Y. Simsek, \textit{Apostol type Daehee numbers and
polynomials}, Adv. Studies Contemp. Math. \textbf{26}(3) (2016), 1--12.

\bibitem{simsekCogent} Y. Simsek, \textit{Analysis of the }$p$\textit{-adic }%
$q$\textit{-Volkenborn integrals: An approach to generalized Apostol-type
special numbers and polynomials and their applications}, Cogent Math. 
\textbf{3} (2016), 1269393, http://dx.doi.org/10.1080/23311835.2016.1269393

\bibitem{simsek2017ascm} Y. Simsek, \textit{Identities on the Changhee
numbers and Apostol-Daehee polynomials}, to appear Adv. Stud. Contemp. Math.
(2017).

\bibitem{Srivastava2011} H. M. Srivastava, \textit{Some generalizations and
basic (or }$q$\textit{-) extensions of the Bernoulli, Euler and Genocchi
polynomials}, Appl. Math. Inf. Sci. \textbf{5} (2011), 390--444.

\bibitem{srivas18} H. M. Srivastava, T. Kim. and Y. Simsek, $q$\textit{\
-Bernoulli numbers and polynomials associated with multiple }$q$\textit{\
-zeta functions and basic }$L$\textit{-series}, Russ. J. Math. Phys. \textbf{%
12} (2005), 241--268.

\bibitem{SrivatavaChoi} H. M. Srivastava and J. Choi, \textit{Zeta and }$q$%
\textit{-Zeta Functions and Associated Series and Integrals}, Elsevier
Science Publishers: Amsterdam, London and New York, 2012.

\bibitem{SrivastavaLiu} H. M. Srivastava and G.-D. Liu, \textit{Some
identities and congruences involving a certain family of numbers}, Russ. J.
Math. Phys. \textbf{16} (2009), 536--542.

\bibitem{Vladimirov} V. S. Vladimirov, I. V. Volovich and E. I. Zelenov, $p$%
\textit{-Adic Analysis and Mathematical Physics}, World Scientific,
Singapore, 1994.

\bibitem{Volkenborn} A. Volkenborn, \textit{On generalized }$p$\textit{-adic
integration}, M\'{e}m. Soc. Math. Fr. \textbf{39-40} (1974), 375--384.

\bibitem{Wang} H. Wang and G. Liu, \textit{An explicit formula for higher
order Bernoulli polynomials of the second}, Integer \textbf{13} (2013),
\#A75.

\bibitem{wikiPEDIAfalling} https://en.wikipedia.org/wiki/Falling\_and%
\_rising\_factorials

\bibitem{wikipe} https://en.wikipedia.org/wiki/Volkenborn\_integral

\bibitem{WikipeLAH} https://en.wikipedia.org/wiki/Lah\_number
\end{thebibliography}
\end{document}